\documentclass[11pt,reqno]{amsart}

\usepackage{amsmath}
\usepackage{amsfonts}
\usepackage{amssymb}
\usepackage{amsthm,color}
\usepackage{amsmath,amsthm,amssymb}
\usepackage{graphicx,mathtools}
\usepackage{cite}
\usepackage{enumitem}

\newtheorem{Theorem}{Theorem}[section]

\newtheorem{Lemma}[Theorem]{Lemma}
\newtheorem{Corollary}[Theorem]{Corollary}

\def\R{{\mathbb R}}

\def\({\left(}
\def\){\right)}

\usepackage{stmaryrd}
\newcommand{\jump}[1]{\llbracket #1 \rrbracket}

\begin{document}

\title[Factorization method for biharmonic scattering]{Factorization method for the biharmonic scattering problem for an absorbing penetrable scatterer}

\author{Rafael Ceja Ayala}
\address{School of Mathematical and Statistical Sciences, Arizona State University, Tempe, AZ 85287}
\email{rcejaaya@asu.edu}

\author{Isaac Harris}
\address{Department of Mathematics, Purdue University, West Lafayette, IN 47907}
\email{ harri814@purdue.edu}

\author{General Ozochiawaeze}
\address{Department of Mathematics, Purdue University, West Lafayette, Indiana, 47907, USA} 
\email{oozochia@purdue.edu}

\subjclass[2010]{35R30, 35R60}

\keywords{inverse scattering, absorbing scatterer, elastic plates, biharmonic wave equation, factorization method, sampling methods, born approximation}

\begin{abstract}
This work extends the factorization method to the inverse scattering problem of reconstructing the shape and location of an absorbing penetrable scatterer embedded in a thin infinite elastic (Kirchhoff--Love) plate. With the assumption that the plate thickness is small compared to the wavelength of the incident wave, the propagation of flexural perturbations is modeled by the two--dimensional biharmonic wave equation in the frequency domain. Within this setting, we provide a rigorous justification of the factorization method and demonstrate that it yields a binary criterion for distinguishing whether a sampling point lies inside or outside the scatterer, using only the spectral data of the far--field operator. In addition, we numerically analyze the Born approximation for weak scatterers in this biharmonic scattering context and compute the relative error against exact far--field data for sample weak scatterers, thereby quantifying its validity as a limited but useful approximation. 
\end{abstract}

\maketitle

\section{Introduction}
This paper investigates the inverse scattering problem of reconstructing a penetrable, absorbing scatterer embedded in a thin infinite elastic (Kirchhoff--Love) plate. This implies that the wave propagation is governed by the biharmonic `Helmholtz' equation in the frequency domain. Recently, the study of biharmonic scattering problems has garnered significant attention due to their widespread applications, including the design of ultra--broadband elastic cloaks for vibration control in vehicles and earthquake--resistant passive systems for smart buildings \cite{PhysRevB.85.020301, PhysRevLett.103.024301, PhysRevB.79.033102}. Platonic crystals, which are periodic arrays of cavities, enable wave manipulation similar to photonic and phononic crystals \cite{farhat2014platonic, gao2018theoretical}. Acoustic black holes passively trap flexural waves, supporting applications in noise reduction, energy harvesting, and biomedical devices (see \cite{PELAT2020115316}). The biharmonic model is additionally important for non-destructive testing and structural health monitoring in the aerospace industry \cite{MEMMOLO2018568}. Recently, it has been shown that the similar biharmonic equations can be used to study sea ice and ice shelves in \cite{seaice1,seaice2}. In contrast to classical scattering problems involving wave propagation in unbounded media, biharmonic waves satisfy a fourth--order boundary value problem and describe out--of--plane displacements in thin elastic plates, presenting unique mathematical and computational challenges. Compared with the second-order differential equations like the acoustic, elastic, and electromagnetic waves, the research of the inverse scattering problem of the biharmonic wave equation is not as extensive and this work extends the literature on biharmonic scattering. The increase in order leads to unique challenges in the extension of computational approaches to the inverse scattering of biharmonic waves.

In this contribution, we provide the first extension of the factorization method to the inverse scattering problem for biharmonic waves. It has been shown in \cite{CejaAyalaHarrisSanchezVizuet2025} that the corresponding scattering problem is well--posed provided the imaginary part of the `refractive index' is positive in the scatterer. This recent work also establishes the reciprocity relationships for both near--field and far--field data.  The well--posedness is shown by considering the equivalent variational form and proving that it has the Fredholm property, then using the absorbing assumption to guarantee uniqueness. Therefore, we will assume that the scatterer is absorbing throughout the paper.  Building on this foundation, we consider the corresponding inverse shape problem of recovering the scatterer from the measured far--field data via the factorization method. To prove that the factorization method can recover the biharmonic scatterer we will employ techniques developed for acoustic scatterers which need to be modified for this new model. 

The factorization method falls under the category of sampling methods, which are non--iterative methods that recover the shape and/or location of the scatterer without requiring prior knowledge of its material properties or a forward solver. Among them, the factorization method, originally introduced by Kirsch in \cite{Kirsch1998} and the monograph \cite{kirsch2008factorization} for the inverse acoustic scattering problem, continues to be extensively studied. Compared with it's predecessor (the linear sampling method \cite{ColtonKirsch1996}), the factorization method give that the scatterer can be obtained via a range test using the far--field operator to characterize if a point is inside or outside the scatterer. This can be done in computable manner that provides a sufficient and necessary criterion for the qualitative reconstruction. The factorization method stands out for its rigorous characterization of the support of the scatterer in terms of the spectral properties of a self-adjoint form of the far--field operator. This framework provides a precise link between the analytic structure of the far--field data and the geometry of the scatterer. The factorization method has been studied for a wide range of scatterers, including sound-soft, sound-hard, and impedance obstacles, as well as acoustic, electromagnetic, and elastic scattering problems (see, e.g., \cite{cakoni2014qualitative, kirsch2008factorization, HuKirschSini2012, CharalambopoulosEtAl2006}). Despite this rich body of work, the biharmonic wave setting has not yet been addressed within the factorization method framework. 

For biharmonic waves, other sampling methods have only recently been explored. The linear sampling method has been applied to the reconstruction of clamped cavities in a thin plate using both near--field and far--field data \cite{bourgeois2020linear, GuoEtAl2024, HarrisLiOzochiawaeze}. The work presented in \cite{HarrisLiOzochiawaeze}, also adapts the extended sampling method to recover the location of clamped cavities with limited aperture data and multifrequency measurements for improved resolution. This demonstrates the adaptability of sampling--based techniques in practical measurement scenarios. The direct sampling method has also been extended to the biharmonic scattering problem for clamped cavities with the added benefit of their stability with respect to noisy data and precise resolution analysis (see \cite{HLP-directsample}). Beyond these sampling--based approaches, reverse time migration has also been implemented for biharmonic scattering problems under a variety of boundary conditions \cite{ZhuGe2025}. However, all of these works are limited to impenetrable obstacles; to date, no sampling method has been developed for the inverse shape problem associated with a biharmonic penetrable scatterer.

To bridge this theoretical gap, in the present paper, we consider an inverse biharmonic scattering problem of determining the shape and location of the support of a penetrable scatterer. We extend the factorization method to solve this inverse shape problem. We present a theoretical factorization of the far--field operator of a penetrable absorbing scatterer and rigorously study the operators associated with the factorization. In addition to this theoretical development, we also investigate the `Born approximation' for this biharmonic problem in case of the weak scatterers numerically. In our numerical experiments, we show that this can be used to compute the far--field pattern just as in the case of an acoustic scatterer. Notice that since the Born approximation holds for our biharmonic scattering problem we should have that the Multiple Signal Classification (MUSIC) algorithm can be used to recover small/weak scatterers (see, e.g., \cite{MUSIC-ammari-scattering,Kirsch2002}). Indeed, it is seen in \cite{Kirsch2002} that the MUSIC algorithm is just a discretized version of the factorization method. Even though this connection is evident, we only consider the factorization method in this study. 

The outline of this paper is as follows. In Section \ref{section_2}, we mathematically formulate the model of the inverse biharmonic scattering problem for a penetrable, absorbing medium and define our notations used throughout the paper. In Section \ref{section_3}, the factorization form of the far--field operator is derived and properties of the involved auxiliary operators are presented. Moreover, we establish a theoretical framework for the factorization method. Numerical examples are presented to illustrate the feasibility and effectiveness of our method in Section \ref{section_4}.  To this end, we show that the Born approximation can be used to compute synthetic far--field data the approximation to the exact far-field data computed via separation of variables for a small disk--shaped scatterer. Finally, the paper is concluded with some general remarks and directions for future research in Section \ref{section_conclusion}.

\section{Problem Formulation}\label{section_2}
In this section, we discuss the associated direct scattering problem. This problem models scattering in a thin elastic plate containing a penetrable subregion. Therefore, consider a bounded open set $D\subset \mathbb R^2$ with $C^2-$smooth boundary $\partial D$ representing a scatterer embedded in an infinitely thin, two--dimensional elastic plate governed by the Kirchhoff--Love model in the pure bending regime. Assume that the exterior domain $\mathbb R^2\setminus\overline D$ is simply--connected. We consider the scattering by an incident field in the form of a time--harmonic plane wave
\[
u^i(x,d)=\text{e}^{ \text{i} \kappa x\cdot d},\quad x\in \mathbb R^2,
\]
where $\kappa>0$ represents the wavenumber and $d\in\mathbb S^1$ is a unit vector indicating the direction of wave incidence.
The interaction of the incident field and the scatterer denoted by $D$ produces a total wavefield $u$ which is the superposition of the incident field $u^i$ and a radiating scattered field $u^s\in H_{\text{loc}}^2(\mathbb R^2)$ that satisfies
$$(\Delta^2-\kappa^4n(x))u= 0 \quad \text{in }\mathbb R^2.$$
Since we have that the incident field solves that biharmonic `Helmholtz' equation
$$(\Delta^2-\kappa^4)u^i= 0 \quad \text{in }\mathbb R^2$$
we obtain that the radiating scattered field $u^s$ satisfies 
\begin{align}\label{eqn1}
(\Delta^2-\kappa^4n(x))u^s&=\kappa^4(n(x)-1)u^i \quad \text{in }\mathbb R^2.
\end{align}
Just as in the acoustic scattering case, we assume that the total field $u$ satisfies some transmission conditions on the boundary $\partial D$. Therefore, we assume that that  total wavefield $u$ satisfies
\begin{align}\label{eqn2}
\jump{u}=\jump{\partial_\nu u}=\jump{\Delta u}=\jump{\partial_{\nu}\Delta u}=0,
\end{align}
where, for a given function $\phi$, the notation
\begin{align*} 
\jump{\phi}\coloneqq \phi_{+}\big|_{\partial D}-\phi_{-}\big|_{\partial D}
\end{align*}
denotes the discontinuity across the boundary $\partial D$, and the operators $(\cdot)_{+}\big|_{\partial D}$ and $(\cdot)_{-}\big|_{\partial D}$ represent the trace (or restriction) taken from $\mathbb R^2\setminus\overline D$ and $D$, respectively. In addition, the scattered field and it's are Laplacian satisfy the Sommerfield radiation condition
\begin{align}\label{SRC}
\lim_{r\to\infty}\sqrt{r}(\partial_r u^s-\text{i}\kappa u^s)=0\quad \text{and}\quad \lim_{\to\infty}\sqrt{r}(\partial_r \Delta u^s-\text{i}\kappa \Delta u^s)=0 \quad \text{for $r\coloneqq |x|$ }
\end{align}
which is assumed to be satisfied uniformly in all directions $\hat{x}\coloneqq x/|x|\in \mathbb S^1$. See \cite{CejaAyalaHarrisSanchezVizuet2025} for a detailed discussion of this scattering problem.

For the coefficient $n=n(x)$, we require that 
$$\text{supp}(n(x)-1)=\overline D \quad \text{and } \quad n\in L^{\infty}(D).$$ 
For acoustic scattering problems, $n$ is referred to as the refractive index associated with the scatterer $D$. Moreover, we assume that the scatterer is ``absorbing" which implies that the refractive index $n$ satisfies 
\begin{align}\label{absorbing_cond}
    \Im{(n(x))}&\geq \alpha>0\quad \text{a.e. }x\in D.
\end{align}
This is to insure solvability of the direct scattering problem \eqref{eqn1}--\eqref{SRC} which was established in \cite{CejaAyalaHarrisSanchezVizuet2025}. This would imply that for any fixed incident direction $d$ on the unit circle we have a corresponding scattered field $u^s( \cdot \, , d)$.

Recall, that we wish to consider the inverse problem of recovering the scatterer fro the far--field data. To this end, we will assume that we have access to the measured the far--field pattern
\[
u^{\infty}(\hat{x},d) \quad \text{ for all } \,\, \, \hat{x},d\in\mathbb S^1
\]
corresponding to the forward problem \eqref{eqn1}--\eqref{SRC}. Due to the fact that scattered field $u^s$ is radiating, we have that it has the asymptotic behavior 
\[
u^s(x)=\gamma \frac{\text{e}^{\text{i}\kappa r}}{\sqrt{r}}u^{\infty}(\hat{x},d)+O\left(\frac{1}{r^{3/2}}\right),\quad \text{as } r\to\infty
\]
where the constant 
$$ \gamma = \frac{\text{e}^{\text{i}\pi/4}}{\sqrt{8\pi \kappa}} .$$
We will make use of the fact that the scattered field 
$$u^s = u_{\text{H}} + u_{\text{M}} \quad \text{in $\mathbb R^2\setminus\overline D$}$$ 
where the propagative field $u_{\text{H}}$ and the evanescent field $u_{\text{M}}$ which satisfy the Helmholtz equation (wave number $\kappa$) and the modified Helmholtz equation (wave number $\text{i}\kappa$)
\begin{align}\label{biharm_split}
    (\Delta+\kappa^2)u_{\text{H}}=0 \quad \text{and}\quad (\Delta-\kappa^2)u_{\text{M}}=0 \quad \text{in }\mathbb R^2\setminus\overline D,
\end{align}
along with the Sommerfield radiation conditions
\begin{align}
    \lim_{r\to\infty}\sqrt{r}(\partial_r u_{\text{H}}-\text{i} \kappa u_{\text{H}})=0 \quad \text{and} \quad \lim_{r\to\infty}\sqrt{r}(\partial_r u_{\text{M}}-\text{i} \kappa u_{\text{M}})=0.
\end{align}
It is clear that this decomposition of the scattered field is accomplished by rewriting the biharmonic `Helmholtz' operator using the factorization 
$$(\Delta^2-\kappa^4)=(\Delta-\kappa^2)(\Delta+\kappa^2)=(\Delta+\kappa^2)(\Delta-\kappa^2).$$ 
From this we see that one can define the auxiliary components of the scattered field i.e. the propagative field $u_{\text{H}}$ and evanescent field $u_{\text{M}}$ by
\begin{align}\label{aux_components}
    u_{\text{H}}=-\frac{1}{2\kappa^2}(\Delta u^s-\kappa^2u^s) &\quad \text{and}\quad u_{\text{M}}=\frac{1}{2\kappa^2}(\Delta u^s+\kappa^2u^s),
\end{align}
which implies that 
\begin{align*}
    u^s=u_{\text{H}}+u_{\text{M}} &\quad \text{and}\quad \Delta u^s=-\kappa^2(u_{\text{H}}-u_{\text{M}}).
\end{align*}
By Proposition 2.2 in \cite{bourgeois2020well}, we have that the  auxiliary components have the respective asymptotic behavior as $r\to\infty$ given by 
\begin{align}
    |u_{\text{H}}|=O\left(\frac{1}{\sqrt r}\right)&\quad \text{and}\quad |u_{\text{M}}|=O\left(\frac{\text{e}^{-\kappa r}}{\sqrt{r}}\right).
\end{align}
Note, that it is also known that $\partial_r u_{\text{M}}$ will decay exponentially as $r \to \infty$ see for e.g. \cite{HLP-directsample}.

The asymptotic behavior of the the propagative field $u_{\text{H}}$ and the exponential decay of the evanescent field $u_{\text{M}}$ , implying that the far--field pattern for the scattered field is given by 
$$u^{\infty}(\hat{x},d)=u_{\text{H}}^{\infty}(\hat{x},d)$$ 
and therefore
\begin{align}\label{green_formula}
    u^{\infty}(\hat{x},d)&=\int_{\partial B_R}\left(u_{\text{H}}(y,d)\partial_{\nu(y)}\text{e}^{-\text{i}\kappa \hat{x}\cdot y}-\partial_{\nu(y)} u_{\text{H}} (y,d)\text{e}^{-\text{i}\kappa \hat{x}\cdot y}  \right)\,\text{d}s(y).
\end{align}
Notice, that the far--field data does not contain any information about the function $u_{\text{M}}$ as it decays exponentially. Even though the far--field pattern does not retain any information on the evanescent component $u_{\text{M}}$. From this we see that, given far--field data from this biharmonic scattering problem is equivalent to the far--field pattern for the Helmholtz equation. With this in mind, it would seem that similar techniques that worked for acoustic scattering can still be valid for biharmonic scattering problem. Note that is has been exploited in \cite{HLP-directsample,HarrisLiOzochiawaeze} and we wish to do the same for our problem. 

In order to continue, we recall the function
\begin{align}
    G(x,z;\kappa)&=\frac{1}{2\kappa^2}\left[\Phi_{\text{H}}(x,z;\kappa)-\Phi_{\text{M}}(x,z;\kappa)\right] \quad \text{for }x\neq z
\end{align}
is the radiating fundamental solution for the biharmonic wave operator (i.e. biharmonic `Helmholtz' equation) which is given by $(\Delta^2-\kappa^4)$, where $\Phi_{\text{P}}(x,z;\kappa)$ denotes the fundamental solution of the corresponding Helmholtz or modified Helmholtz equation, with
\begin{align*}
    \Phi_{\text{P}}(x,z;\kappa)&=\begin{dcases}
        \frac{\text{i}}{4}H_0^{(1)}(\kappa|x-z|),\quad \text{P}=\text{H},\\
        \frac{\text{i}}{4}H_0^{(1)}(\text{i}\kappa|x-z|),\quad \text{P}=\text{M}.
    \end{dcases}
\end{align*}
Here $H_{\ell}^{(1)}$ denotes the Hankel function of the first kind of order $\ell$. Due to the exponential decay of the evanescent component $\Phi_{\text{M}}(x,z;\kappa)$, the asymptotic behavior of the propagating component $\Phi_{\text{H}}(x,z;\kappa)$  as $r=|x|\to\infty$ we have that the far--field pattern for the fundamental solution $G(x,z;\kappa)$ is given by
\[
G^{\infty}(\hat{x},z;\kappa)=\frac{1}{2\kappa^2}\text{e}^{-\text{i}\kappa \hat{x}\cdot z}.
\]
Given the incident field $u^i$, the direct scattering problem for (\ref{eqn1})--(\ref{SRC}) is to find the scattered field $u^s$ for a known penetrable medium $(n,D)$. Both the uniqueness of the direct problem and its characterization as Fredholm of index zero were established via a variational approach in \cite{CejaAyalaHarrisSanchezVizuet2025}. The uniqueness proof follows the approach used for acoustic scattering, but in the biharmonic setting the ``absorbing” assumption is essential, whereas it is not in the acoustic case. The inverse scattering problem is to recover the shape and location of the support $D$ given $u^{\infty}(\hat{x},d)$ for all observation directions $\hat{x}\in\mathbb S^1$ and incident directions $d\in\mathbb S^1$. 

We recall the integration--by--parts formulas important to our analysis (see \cite{CejaAyalaHarrisSanchezVizuet2025,HarrisLiOzochiawaeze}). Assume that a bounded domain $\Omega$ is $C^2-$smooth with $\nu$ the outward unit normal. For $v,w\in H^2(\Omega,\Delta^2)$, we have
\begin{align}\label{biharmonic_Green_1}
    \int_{\Omega}(v\Delta^2w-\Delta v\Delta w)\,\text{d}x&= \int_{\partial \Omega}(v\partial_\nu\Delta w-\Delta w\partial_\nu v)\,\text{d}s
\end{align}
and
\begin{align}\label{biharmonic_Green_2}
    \int_{\Omega}(v\Delta^2w-w\Delta^2v)\,\text{d}x&=\int_{\partial \Omega}(\partial_\nu w\Delta v-w\partial_\nu \Delta v+v\partial_\nu\Delta w-\Delta w\partial_\nu v)\,\text{d}s.
\end{align}
Here, the Hilbert Space $H^2(\Omega,\Delta^2)$ is defined by 
$$H^2(\Omega,\Delta^2)\coloneqq \{v\in H^2(\Omega): \Delta^2v\in L^2(\Omega)\}$$
with the associated graph norm/inner--product. The boundary integrals in \eqref{biharmonic_Green_1}--\eqref{biharmonic_Green_2} are understood as a dual pairing from the appropriate Sobolev spaces. It is clear that these formulas are derived directly from Green's first and second theorem. To this end, we derive a Lippmann--Schwinger--type representation of the scattered field. Using the integration--by--parts formula in (\ref{biharmonic_Green_2}) when $x$ is in the interior of $D$ gives the that
\begin{align*}
   u^s(x)\chi_D&=\int_D G(x,z;\kappa)[\kappa^4 u^s(z)-\Delta^2 u^s(z)] \, \text{d}z\\
&\hspace{0.5in}+\int_{\partial D}  \Delta G(x,z;\kappa) \partial_\nu u^s_-(z)-u^s_-(z)\partial_\nu \Delta G(x,z;\kappa)\, \text{d}s(z)\\
&\hspace{0.5in}-\int_{\partial D}  \Delta  u^s_-(z)\partial_{\nu} G(x,z;\kappa)-G(x,z;\kappa)\partial_{\nu}\Delta u^s_-(z)\, \text{d}s(z),
\end{align*}
where $\chi_D$ is the indicator function on the scatterer $D$. In a similar manner, by again using integration--by--parts formula in (\ref{biharmonic_Green_2}) when $x$ is in the exterior of $D$ gives that
\begin{align*}
    u^s(x)(1-\chi_D)=&-\int_{\partial D} \Delta G(x,z;\kappa)\partial_\nu u^s_+(z)-u^s_+(z)\partial_\nu \Delta G(x,z;\kappa)\, \text{d}s(z)\\
    &\hspace{0.5in}+\int_{\partial D} \Delta u^s_+(z)\partial_\nu G(x,z;\kappa)-G(x,z;\kappa)\partial_\nu\Delta u^s_+(z)\, \text{d}s(z)\\
     &\hspace{0.5in}+\int_{\partial B_R}  \Delta G(x,z;\kappa)\partial_\nu u^s(z)-u^s(z)\partial_\nu \Delta G(x,z;\kappa)\, \text{d}s(z).
\end{align*}
where $B_R=\{ x \in \R^d \, \, : \, \, |x|<R \}$ such that $D \subset B_R$. Therefore, by adding the above expressions, we obtain the Lippmann--Schwinger--type representation of the scattered field 
\begin{align}\label{LSsf}
u^s(x,d)&=\kappa^4\int_D (n(z)-1)G(x,z)(u^s(z,d)+u^i(z,d))\,\text{d}z
\end{align}
where we make the dependance of the incident direction explicit. A similar representation was also used in \cite{timedep-biharwave}. Notice, that we have used the fact that the scattered field and fundamental solution satisfying the radiation condition \eqref{SRC} to handle the boundary integral over $\partial B_R$ by letting $R \to \infty$. Equation \eqref{LSsf} will be vital for the analysis in later sections.

\section{Factorization Method Framework}\label{section_3}
\subsection{Factorization of the far--field operator}
This section is dedicated to constructing a suitable factorization of the far--field operator so that we can use the theoretical results in \cite{kirsch2008factorization}. This will allow us to determine a suitable indicator function for the support $D$ in terms of far--field measurements. The main goal of the factorization method is to connect the support $D$ to the range of an operator defined by the measured far--field operator. We make this connection by analyzing a derived factorization of the far--field operator developed in this section.

We introduce the far--field operator $F: L^2(\mathbb S^1)\to L^2(\mathbb S^1)$ defined by
\begin{align}\label{farfield_op}
Fg(\hat{x})\coloneqq \int_{\mathbb S^1}u^{\infty}(\hat{x},d)g(d)\,\text{d}s(d),\quad g\in L^2(\mathbb S^1),\, \hat{x}\in\mathbb S^1.
\end{align}
Recall, that we assume that the refractive index satisfies $n\in L^{\infty}(D)$  along with the absorbing scatterer condition $\Im{(n)}\geq \alpha>0$ for some constant $\alpha>0$. In order to derive a suitable factorization of the far--field operator, we consider an auxiliary problem for a generic incident field. To this end, we let $f\in L^2(D)$ and define the auxiliary problem for the unique radiating solution $w\in H_{\text{loc}}^2(\mathbb R^2)$ satisfying
\begin{align}\label{biharmonic_bvp}
    \begin{dcases}
        (\Delta^2-n(x)\kappa^4)w=\kappa^4(n(x)-1)f\quad \text{in }\mathbb R^2,\\
        \jump{w}=\jump{\partial_\nu w}=\jump{\Delta w}=\jump{\partial_{\nu}\Delta w}=0, \quad \text{on } \partial D\\
        \lim_{r\to\infty}\sqrt{r}(\partial_r w-\text{i}\kappa w)= \lim_{r\to\infty}\sqrt{r}(\partial_r \Delta w-\text{i}\kappa \Delta w)=0
    \end{dcases}
\end{align}
where again $r=|x|$. By linearity of the direct scattering problem, we have that $Fg$ is the far--field pattern of $w$ solution of (\ref{biharmonic_bvp}) with $f=v_g$ in $D$ by the superposition principle, where
\begin{align*}
    v_g(x)&\coloneqq \int_{\mathbb S^1}\text{e}^{\text{i}\kappa {x}\cdot d}g(d)\,\text{d}s(d),\quad g\in L^2(\mathbb S^1),\quad \text{for any} \,\, x\in \mathbb R^2
\end{align*}
is the Herglotz wave function. Now consider the (compact) Herglotz operator $H: L^2(\mathbb S^1)\to L^2(D)$ defined by 
\begin{align}\label{herglotz_op}
    Hg&\coloneqq v_g|_{D},
\end{align}
and associated source-to-far--field operator $G: L^2(D)\to L^2(\mathbb S^1)$ defined by 
\begin{align*}
    Gf&\coloneqq w^{\infty},
\end{align*}
where $w^{\infty}$ is the far--field pattern of $w\in H_{\text{loc}}^2(\mathbb R^2)$ solution of (\ref{biharmonic_bvp}). Then, by the linearity of the auxiliary problem \eqref{biharmonic_bvp} along with the superposition principle, we obtain that 
\[F=GH.\]
We now establish a factorization of the far--field operator for the factorization method. 

To proceed, for the case under consideration, 
we have the integral identity Lippmann--Schwinger--type representation (\ref{LSsf}) applied to the auxiliary field $w$, implies that 
\begin{align}\label{scattered_field_integral}
    w(x)&=\kappa^4\int_D \left(n(y)-1\right)G(x,y)\left(w(y)+f(y)\right)\,\text{d}y.
\end{align}
Using the fact that 
$$G(x,y)=\gamma \frac{\text{e}^{\text{i}\kappa r}}{\sqrt{r}} \cdot \frac{1}{2 \kappa^2} \text{e}^{-\text{i}\kappa \hat{x}\cdot y}+O\left(\frac{1}{r^{3/2}}\right),\quad \text{as } r=|x|\to\infty$$ 
we can derive an integral representation of the far--field pattern. 
Therefore, we have that the far--field pattern of $w$ has the following expression 
\[
w^{\infty}(\hat{x})=\frac{\kappa^2}{2}\int_{D} \text{e}^{-\text{i}\kappa \hat{x}\cdot y}\left(n(y)-1\right) \left(f(y)+w(y)\right)\,\text{d}y.
\]
Therefore, one simply has $G=H^{*}T$ where $H^{*}: L^2(D)\to L^2(\mathbb S^1)$ is the adjoint of $H$ given by 
\begin{align}
    H^{*}\varphi(\hat{x})\coloneqq \int_{D}\text{e}^{-\text{i}\kappa \hat{x}\cdot y}\varphi(y)\,\text{d}y,\quad \varphi\in L^2(D),\quad \hat{x}\in \mathbb S^1,
\end{align}
and where $T: L^2(D)\to L^2(D)$ is defined by
\begin{align}\label{T}
    Tf&\coloneqq \frac{\kappa^2}{2}(n-1)(f+w),
\end{align}
with auxiliary field $w\in H_{\text{loc}}^2(\mathbb R^2)$ being the unique solution to (\ref{biharmonic_bvp}). 

Due to the well--posedness estimate in \cite{CejaAyalaHarrisSanchezVizuet2025} for the auxiliary scattering problem (\ref{biharmonic_bvp}) it is clear that $T$ is a well defined bounded linear operator. From this we see just as in the acoustic scattering case $G$ is compact by the compactness of the Herglotz operator $H$ and the decomposition $G=H^{*}T$. Thus, we obtain the following factorization for the far--field operator $F$.
\begin{Lemma}
    The far-field operator $F: L^2(\mathbb S^1)\to L^2(\mathbb S^1)$ for the scattering problem (\ref{eqn1})-(\ref{SRC}) has the factorization 
    \begin{align}\label{factorization}
    F&=H^{*}TH.
\end{align}
\end{Lemma}
Notice that for our biharmonic scattering problem, the operator $H$ is defined just as in the isotropic acoustic scattering case but the operator $T$ is defined similarly to the acoustic case but is analytically different due to it connection to \eqref{biharmonic_bvp}. We also recall the following characterization of all the points in $D$ by the range of the operator $H^{*}$.
\begin{Theorem}\label{range_char_D}
    Let $\overline D=\text{supp}(n-1)$ be open and bounded so that the complement $\mathbb R^2\setminus\overline D$ is connected. For $z\in\mathbb R^2$ define $\phi_z\in L^2(\mathbb S^1)$ by
    \begin{align}\label{test_vector}
        \phi_z(\hat{x})&\coloneqq \mathrm{e}^{-\mathrm{i}\kappa \hat{x}\cdot z},\quad \hat{x}\in \mathbb S^1.
    \end{align}
    Then we have
    \begin{align*} 
    z\in D \quad\text{ if and only if}\quad \phi_z\in \text{Range} (H^{*}).
    \end{align*}
\end{Theorem}
The proof of Theorem \ref{range_char_D} can be found in \cite{kirsch2008factorization} Theorem 4.6. To avoid repetition, we will not provide the proof here. With this we only need to study the analytical properties of the operator $T$ defined by \eqref{T} to establish that factorization method. 

\subsection{Analysis of the operator $T$}
Now, we analyze the properties of the middle operator $T$ defined (\ref{T}) to show that the factorization method is applicable to our scattering problem. The operator $T$ is defined similarly to the operator used in the factorization of the far--field operator for the acoustic scattering problem see for e.g. \cite{Kirsch2002}. Provided the material of the support $D$ is absorbing, we study the imaginary part of the operator $T$. To this end, we define 
$\Im{(T)}$ by
\[
\Im{(T)}\coloneqq \frac{1}{2\text{i}}(T-T^{*}),
\]
where $T^{*}$ is the adjoint of $T$ in $L^2(D)$. Following \cite{Kirsch2002}, we can prove that $\Im{(T)}$ is a coercive operator due to the absorbing scatterer assumption. This is done by appealing to (\ref{biharmonic_bvp}) as well as the integration--by--parts formulas  given in \eqref{biharmonic_Green_1}--\eqref{biharmonic_Green_2}.

\begin{Theorem}\label{imag_T_coercive_thm}
    Assume that $\Im{(n)}$ is uniformly positive in $D$. Then the operator $\Im{(T)}: L^2(D)\to L^2(D)$ is coercive. 
\end{Theorem}
\begin{proof}
    We start by deriving a useful identity related to the imaginary part of $T$ using the direct scattering problem. To this end, for simplicity we let $(\,,\,)$ denoting the $L^2(D)$ scalar product. Now, we let  $f\in L^2(D)$ and $w\in H_{\text{loc}}^2(\mathbb R^2)$ be the corresponding solution to (\ref{biharmonic_bvp}). By definition of the operator $T$ in \eqref{T} 
    \begin{align}\label{Tscalar_prod}
        (Tf,f)&= \frac{\kappa^2}{2}\int_{D}(n-1)(w+f)\overline f\,\text{d}x.
    \end{align}
    Noting that $\overline f=\overline{f+w}-\overline w$, we can rewrite the scalar product
    \[
    (Tf,f)=\frac{\kappa^2}{2}\int_D (n-1)|w+f|^2\,\text{d}x-\frac{\kappa^2}{2}\int_D (n-1)(w+f)\overline w\,\text{d}x.
    \]
    
    Recall, that by the auxiliary scattering problem \eqref{biharmonic_bvp} we have that 
    \begin{align*}
    \Delta^2w-\kappa^4w&=\Delta^2w-\kappa^4nw+\kappa^4nw-\kappa^4w\\
    &=\kappa^4(n-1)f+\kappa^4(n-1)w\\
    &=\kappa^4(n-1)(f+w).
    \end{align*}
    Therefore, we obtain
    \begin{align*}
        (Tf,f)&=\frac{\kappa^2}{2}\int_D (n-1)|w+f|^2\,\text{d}x-\frac{1}{2\kappa^2}\int_D\overline{w}(\Delta^2w-\kappa^4w)\,\text{d}x.
    \end{align*}
    Since $n=1$ outside $D$, it is clear that 
    \[
    \Delta^2w-\kappa^4w=0 \quad \text{in }B_R\setminus\overline D.
    \]
    with $B_R$ a ball of radius $R$ containing $D$.
    Thus, by integrating by parts over $B_R$ using the integral identity in (\ref{biharmonic_Green_1}), we obtain
    \begin{align*}
    (Tf,f)&=\frac{\kappa^2}{2}\int_D (n-1)|w+f|^2\,\text{d}x-\frac{1}{2\kappa^2}\int_{B_R}\overline{w}(\Delta^2w-\kappa^4w)\,\text{d}x\\
    &=\frac{\kappa^2}{2}\int_D (n-1)|w+f|^2\,\text{d}x-\frac{1}{2\kappa^2}\left[ \int_{B_R} |\Delta w|^2-\kappa^4|w|^2\,\text{d}x+\int_{\partial B_R} \overline w\partial_{r} \Delta w-\partial_r\overline w\Delta w\,\text{d}s \right].
    \end{align*}
    
    In order to proceed, we need to study the boundary integral so we define
    \begin{align*}
        I(R)&\coloneqq \int_{\partial B_R}\overline w\partial_r\Delta w-\partial_r\overline w\Delta w\,\text{d}s.
    \end{align*}
    Because the radiating solution $w$ has the decomposition $w=w_{\text{H}}+w_{\text{M}}$ and $\Delta w=-\kappa^2(w_{\text{H}}-w_{\text{M}})$ with their asymptotic behaviors
    \begin{align*}
        \quad |\partial_r w_{\text{H}}- \text{i} \kappa w_{\text{H}}|=O(r^{-3/2})
        \quad \text{along with}\quad |w_{\text{M}}|=O(r^{-1/2}\text{e}^{-\kappa r}), \quad \text{and} \,\, |\partial_r w_{\text{M}}|=O(r^{-1/2}\text{e}^{-\kappa r})
    \end{align*}
   as $r\to\infty$ uniformly with respect to $\hat{x}\in \mathbb S^1$, we obtain
    \begin{align*}
        I(R)&= \int_{\partial B_R}\overline w\partial_r\Delta w-\partial_r\overline w\Delta w\,\text{d}s\\
        &=-\kappa^2\int_{\partial B_R}(w_{\text{H}}\partial_r\overline w_{\text{H}}-\overline w_{\text{H}}\partial_r w_{\text{H}})\,\text{d}s + o(1)\\
        &=-2\text{i} \kappa^3|\gamma|^2\int_{\mathbb S^1}|w_{\text{H}}^{\infty}(\hat{x})|^2\,\text{d}s(\hat{x}) + o(1)
    \end{align*}
    as $R\to\infty$. Hence, we obtain the identity
    \[
    (Tf,f)=\frac{\kappa^2}{2}\int_D (n-1)|w+f|^2\,\text{d}x-\frac{1}{2\kappa^2}\int_{B_R}(\Delta w|^2-\kappa^4|w|^2)\,\text{d}x+\text{i} \kappa |\gamma|^2\int_{\mathbb S^1}|w_{\text{H}}^{\infty}(\hat{x})|^2\,\text{d}s(\hat{x}).
    \]
    up to leading order as $R\to\infty$. Taking the imaginary part yields
    \begin{align}\label{imag_T}
    \Im{(Tf,f)}=\frac{\kappa^2}{2}\int_D \Im{(n)}|w+f|^2\,\text{d}x+\kappa |\gamma|^2\int_{\mathbb S^1}|w_{\text{H}}^{\infty}(\hat{x})|^2\,\text{d}s(\hat{x}).
    \end{align}
    We are now in the position to prove the coercivity property. 
    
    To prove coercivity, we proceed by way of contradiction. Therefore, assume that there existence of a sequence $f_{\ell}\in L^2(D)$ such that
    \begin{align*}
        \|f_{\ell}\|_{L^2(D)}=1 \quad \text{for all $\ell \in \mathbb{N}$} \quad \text{and } \quad \Im{(Tf_{\ell},f_{\ell})}\to 0 \quad \text{as $\ell \to \infty$}.
    \end{align*}
   We denote by $w_{\ell}\in H_{\text{loc}}^2(\mathbb R^2)$ solution of (\ref{biharmonic_bvp}) with $f=f_{\ell}$. Then up to changing the initial sequence, we can assume that $f_{\ell}$ weakly converges to some $f\in L^2(D)$ and $w_{\ell}$ converges weakly in $H_{\text{loc}}^2(\mathbb R^2)$ and strongly in $L^2(D)$ to some $w\in H_{\text{loc}}^2(\mathbb R^2)$. Since we have assumed that
   $$\Im{(Tf_{\ell}, f_{\ell})} \to 0\quad \text{as $\ell \to \infty$}$$ 
   this forces the sequence $(w_{\ell}+f_{\ell}) \to 0$ strongly in $L^2(D)$ as $\ell \to \infty$. It is clear that $w_{\ell}\to 0$ strongly in $L^2(D)$ as $\ell \to \infty$ by the compact embedding of $H^2(D)$ into $L^2(D)$. Note that the sequence $w_{\ell}$ satisfies
   \[
   \Delta^2w_{\ell}-\kappa^4w_{\ell}=\kappa^4(n-1)(f_{\ell}+w_{\ell})\to 0\quad \text{as }\ell\to\infty.
   \]
   Thus, the limiting function $w$ is a radiating solution to
   \[
   \Delta^2w-\kappa^4w=0\quad\text{in }\mathbb R^2.
   \]
   Substituting the decomposition $w=w_{\text{H}}+w_{\text{M}}$ leads to
   \begin{align*}
       (\Delta+\kappa^2)w_{\text{H}}=0,&\quad \text{and }\quad (\Delta-\kappa^2)w_{\text{M}}=0\quad \text{in }\mathbb R^2,
   \end{align*}
   with $w_{\text{H}}$ and $w_{\text{M}}$ radiating. It is well--known that an entire, radiating solution to the Helmholtz equation is identically zero by Rellich's lemma and the unique continuation property, therefore $w_{\text{H}}\equiv 0$. We multiply the equation $(\Delta-\kappa^2)w_{\text{M}}=0$ by $\overline w_{\text{M}}$ and integrate by parts over a ball $B_R$ of radius $R$ containing $D$ to obtain
   \begin{align*}
       \int_{B_R}\overline w_{\text{M}}(\Delta w_{\text{M}}-\kappa^2w_{\text{M}})\,\text{d}x&=0\\
    \int_{B_R}(|\nabla w_{\text{M}}|^2+\kappa^2|w_{\text{M}}|^2)\,\text{d}x&=\int_{\partial B_R}\frac{\partial w_{\text{M}}}{\partial r}\overline w_{\text{M}}\,\text{d}s.
   \end{align*}
   As $R\to\infty$, the right--hand side of the equation approaches $0$ by exponential decay of $w_{\text{M}}$ and $\partial_r w_{\text{M}}$, thus we obtain
   \[
   \int_{\mathbb R^2}(|\nabla w_{\text{M}}|^2+\kappa^2|w_{\text{M}}|^2)\,\text{d}x=0.
   \]
  This would imply that  
   \[
  \| w_{\text{M}}\|_{H^1(\mathbb R^2)}^2=0\quad \text{in }\mathbb R^2.
  \]
   Thus, $ w_{\text{M}}\equiv0$ in $\mathbb R^2$, which implying $w_{\ell}\to 0$ in $L^2(D)$. Therefore, we can obtain $f_{\ell}\to 0$ in $L^2(D)$, contradicting $\| f_{\ell}\|_{L^2(D)}=1$, this proves the claim.
   \end{proof}
We now formulate and prove the main range test identity for the factorization method extended to the biharmonic scattering of an absorbing penetrable medium. We recall the definition of the self-adjoint compact operator 
\[
\Im{(F)}\coloneqq \frac{1}{2\text{i} }(F-F^{*}),
\]
where $F^{*}$ is the $L^2$--adjoint of the far--field operator $F$. Notice, that since $\Im{(T)}$ is coercive we see that the operator $\Im{(F)}$ is positive. This would then give that it has a square root via its spectral decomposition. The following result gives that the scatterer can be characterized by a range test utilizing the known operator $\Im{(F)}$ or more specifically its square root.

\begin{Theorem}\label{Range_Test_F}
    Assume that $\Im{(n)}$ is uniformly positive in $D$. Furthermore, let $\phi_z\in L^2(\mathbb S^1)$ be defined by (\ref{test_vector}).
    Then 
    \begin{center}
    $z\in D$ if and only if $\phi_z\in \text{Range}\big(\Im{(F)}^{1/2}\big)$.
    \end{center}
\end{Theorem}
\begin{proof}
    Noting that
    \[
    \Im{(Tf,f)}= \left(\Im{(T)}f,f \right)>0
    \]
    and we can conclude that by Theorem \ref{imag_T_coercive_thm} that the self--adjoint operator $\Im{(T)}$ is coercive on $L^2(D)$. Furthermore, the factorization (\ref{factorization}) yields the factorization of $\Im{(F)}$ in the form
    \begin{align}\label{imag_factorization}
        \Im{(F)}&=H^{*}\Im{(T)}H.
    \end{align}
   by direct calculations. Now, by appealing to Corollary 1.22 in \cite{kirsch2008factorization} we have that the ranges of $H^{*}$ and $\Im{(F)}^{1/2}$ coincide. The combination of the characterization of $D$ given in Theorem \ref{range_char_D} implies the assertion.
\end{proof}
We now let $\lambda_{j}\in \mathbb R^{+}$ and $\psi_j\in L^2(\mathbb S^1)$
be the orthonormal eigensystem of the positive, self--adjoint compact operator $\Im{(F)}$. Note that the positivity of the eigenvalues comes from the fact that $\Im{(T)}$ is coercive and $H$ being injective. Then applying Picard's range criterion (see for e.g. Theorem 2.7 of \cite{cakoni2014qualitative}) to Theorem \ref{Range_Test_F} gives the following.
\begin{Corollary}\label{Picard_range_thm}
    Assume that the assumptions of Theorem \ref{Range_Test_F} are satisfied. Let $\phi_z\in L^2(\mathbb S^1)$ be defined by (\ref{test_vector}).  Then for all $z\in \mathbb R^2$,
    \begin{align*}
        \sum_{j=1}^{\infty}\frac{|(\phi_z,\psi_j)_{L^2(\mathbb S^1)}|^2}{\lambda_j}<\infty \quad \text{if and only if}\quad z\in D
    \end{align*}
    where $\lambda_{j}\in \mathbb R^{+}$ and $\psi_j\in L^2(\mathbb S^1)$ is the orthonormal eigensystem of the positive, self-adjoint compact operator $\Im{(F)}$.
\end{Corollary}
From this, we see that Corollary \ref{Picard_range_thm} gives a way to characterize the penetrable absorbing scatterer $D$ from the far--field data. Indeed, we see that from the spectral data of the known operator $\Im{(F)}$ we have that 
\begin{align*}
    W(z)&\coloneqq \left[ \sum_{j=1}^{\infty}\frac{|(\phi_z,\psi_j)_{L^2(\mathbb S^1)}|^2}{\lambda_j}    \right]^{-1}>0 \quad \text{if and only if}\quad z\in D.
\end{align*}
Thus, one can reconstruct $D$ by plotting the function $W(z)$ since we have that
\begin{align*}
    \chi_D(z)&= \text{sign}(W(z))=\begin{dcases}
        1,\quad \text{if }z\in D,\\
        0, \quad \text{if }z\notin D
    \end{dcases}
\end{align*}
is the characteristic function for $D$. Notice, that if we had two scatterers that produced the same far--field data, this would imply that the scatterers will produce the same characteristic function. This would then imply that the scatterers are the same, which gives the following result. 

\begin{Corollary}\label{uniq-fromFM}
    The far--field pattern for all $\hat{x}$ and $d$ in $\mathbb{S}^1$ corresponding to \eqref{eqn1}--\eqref{SRC} uniquely determine that scatterer $D$ provided that  $\Im{(n)}$ is uniformly positive in $D$.
\end{Corollary}

Note that we have uniqueness of the inverse shape problem with only far--field data that does not possess any information about the evanescent component to the scattered field. It is shown in \cite{FF-uniqueness} that the biharmonic far--field pattern is enough to recover a clamped obstacle. This is interesting due to the fact the for the biharmonic case, information about the scattered field is lost (unlike the acoustic case) but we still have uniqueness for the inverse problem. 

In the presence of noisy measurements, the indicator function given in Corollary \ref{Picard_range_thm} can be reformulated as the so--called regularized factorization method 
\[
W_{\alpha}(z)\coloneqq \left[\sum_{j=1}^{\infty} \frac{|\lambda_j|}{\alpha+|\lambda_j|^2} \left|(\phi_z, \psi_j)_{L^2(\mathbb S^1)}\right|^2   \right]^{-1},
 \]
using `Tikhonov' regularization. Notice that in the original imaging function $W(z)$ we divide by the eigenvalues of a compact operator. This is not optimal since the eigenvalues tend to zero rapidly in many applications. Therefore, in \cite{regfm} it is proven that 
$$z \in D \quad \text{if and only if } \quad \liminf\limits_{\alpha \to 0^+}W_{\alpha}(z)>0.$$
This result implies that the regularized imaging function $W_{\alpha}(z)$ can be used to recover the scatterer $D$. The regularization parameter $\alpha$ can be chosen ad--hoc or by a discrepancy principle provide one has noisy far--field data. Note, that the analysis of this regularized factorization method was initially studied for diffuse optical tomography and was motivated by \cite{arens-lechleiter,GLSM,RegFM}.


\section{Numerical Examples}\label{section_4}
We now give some numerical reconstructions of penetrable absorbing cavities $D$ using the indicator function developed in Corollary \ref{Picard_range_thm} in two dimensions. To this end, we assume that the scatterer has a small area. With this assumption, we can exploit the Born approximation for the scattered field to simplify the calculations of the synthetic data. Next we will derive a series expansion for a disk of radius $0<\epsilon\leq 1$ centered at the origin, which we will use as a testing scenario to verify the accuracy of the numerical approximation and show that we are able to obtain accurate approximations of the far--field data for weak scatterers $D$. Finally, we explain the implementation of the factorization method and report the successful reconstruction of a variety of penetrable absorbing cavities from knowledge of the far--field data for different parameter settings.

We provide numerical reconstructions using \texttt{MATLAB}. Recall that the scattered field $u^s({x},d)$ is given by the integral representation (\ref{scattered_field_integral}).
We compute the synthetic far--field data using the Born approximation for weak scatterers $D$ such that $|D|\ll 1$. Therefore, just as in the acoustic case this would seem to imply that 
\[
u^{s}({x},d)\approx {\kappa^4}\int_D (n(y)-1) G(x,y) \text{e}^{\text{i}\kappa y\cdot d}\,\text{d}y.
\]
which gives
\[
u^{\infty}(\hat{x},d)\approx \frac{\kappa^2}{2}\int_D (n(y)-1)\text{e}^{-\text{i}\kappa y\cdot (\hat{x}-d)}\,\text{d}y.
\]
In all our examples, for simplicity we will take a constant refractive index in the scatterer. 
In the numerical implementation, to discretize the problem, we compute $u^{\infty}(\hat{x}_i,d_j)$ using numerical integration at $64$ equally spaced points on the unit circle given by
\[
\hat{x}_i=d_i=(\cos{\theta_i},\sin{\theta_i})\quad \text{with }\quad\theta_i=2\pi(i-1)/64.
\]
This gives the discretized far--field operator given by
\[
\mathbf{F}=\left[u^{\infty}(\hat{x}_i,d_j) \right]_{i,j=1}^{64}
\]
which will be used to recover the penetrable scatterer. Therefore, we have that the imaging functional that discretizes the version of Corollary \ref{Picard_range_thm} that we will plot is given by
\[
W_{\text{FM}}(z)=\left[ \sum_{j=1}^{64}\frac{\phi^2(\sigma_j;\alpha)}{\sigma_j} |(\mathbf{u}_j,\boldsymbol{\ell}_z)|^2  \right]^{-1}\quad \text{with }\quad \boldsymbol{\ell}_z=[\text{e}^{-\text{i}\kappa \hat{x}_i\cdot z}]_{i=1}^{64}.
\]
Here $\sigma_j$ are the singular values and $\mathbf u_j$ are the left singular vectors of
\[
\Im{(\mathbf F)}=\frac{\mathbf F-\mathbf F^{*}}{2\text{i}}
\]
and the filter function $\phi(t;\alpha)$ is given by
\[
\phi(t;\alpha)=\frac{t^2}{t^2+\alpha}
\]
for Tikhonov regularization with a fixed regularization parameter $\alpha=10^{-5}$.

To simulate measurement error in the data, random noise is introduced into the multi-static far-field matrix $\mathbf F$, resulting in the following in the following.
\[
\mathbf F^{\delta}=\left[\mathbf F(i,j)(1+\delta \mathbf{R}(i,j)  \right]_{i,j=1}^{64},
\]
where the error matrix $\mathbf R\in\mathbb C^{64\times 64}$ whose entries initally have real and parts consists of random values within the interval $[-1,1]$, and is then normalized. Therefore, we have that $0<\delta<1$ represents the relative noise level added to the computed data.

\subsection{Evaluation of the Born Approximation Against the Exact Solution}
In this section, we investigate how to approximate the biharmonic far--field pattern arising from scattering by a penetrable, absorbing medium. Our focus will be on the case of small disk scatterers, where we compare the exact far--field with the Born approximation to assess the accuracy of the latter approximation. As a concrete example, we consider small disk scatterers and compare their exact far--field response with that predicted by the Born approximation.

To this end, we consider the separation of variables approach outlined in \cite{CejaAyalaHarrisSanchezVizuet2025} and assume that $D=B_{\epsilon}$ (i.e. small disk scatterers), where $0<\epsilon\leq 1$ is the radius of a small disk and $n$ is given by a constant. This implies that the direct scattering problem (\ref{eqn1})--(\ref{SRC}) for the scattered field $u^s(r,\theta)$ outside the scatterer and total field $u(r,\theta)$ inside the scatterer can be written as
\[
\Delta^2u^s-\kappa^4u^s=0\quad \text{in }\mathbb R^2\setminus\overline{B_{\epsilon}}\quad\text{and}\quad \Delta^2u-\kappa^4u=0\quad\text{in }B_{\epsilon}.
\]
Here, denoting by $f(r,\theta)$ the value of the source function $f$ at the point $x=r(\cos{\theta},\sin{\theta})$ in $\R^2$. This given that the boundary conditions are imposed at $r=\epsilon$ and are given by
\[
u^s(\epsilon,\theta)-u(\epsilon,\theta)=-u^i(\epsilon,\theta),\quad \partial_ru^s(\epsilon,\theta)-\partial_ru(\epsilon,\theta)=-\partial_ru^i(\epsilon,\theta)
\]
and
\[
\Delta u^s(\epsilon,\theta)-\Delta u(\epsilon,\theta)=-\Delta u^i(\epsilon,\theta),\quad \partial_r\Delta u^s(\epsilon,\theta)-\partial_r \Delta u(\epsilon,\theta)=-\partial_r \Delta u^i(\epsilon,\theta).
\]
We recall the Jacobi--Anger expansion for the incident plane wave
\[
u^i(r,\theta)=\sum_{|\ell|=0}^{\infty}\text{i}^{\ell}J_{\ell}(\kappa r)\text{e}^{\text{i}\ell (\theta-\phi)}
\]
where the incident direction is given by $d=(\cos{\phi},\sin{\phi})$. With this in mind, we make the ansatz that the propagative $u_{\text{H}}$ and evanescent $u_{\text{M}}$ components of the scattered field $u^s$ can be written in terms of the following  Fourier series expansion
\[
u_{\text{H}}(r,\theta)=\sum_{|\ell|=0}^{\infty}\text{i}^{\ell}a_{\ell}H_{\ell}^{(1)}(\kappa r)\text{e}^{\text{i}\ell(\theta-\phi)}\quad \text{and}\quad u_{\text{M}}(r,\theta)=\sum_{|\ell|=0}^{\infty}\text{i}^{\ell}b_{\ell}H_{\ell}^{(1)}(\text{i}\kappa r)\text{e}^{\text{i}\ell(\theta-\phi)}.
\]
Similarly, it has been shown that the one has a similar decompsition for the total field inside the scatterer. Therefore, we consider the biharmonic wave decomposition of the total field $u=u_{\text{pr}}+u_{\text{ev}}$ such that 
\[
u_{\text{pr}}(r,\theta)=\sum_{|\ell|=0}^{\infty}\text{i}^{\ell}c_{\ell}J_{\ell}(\kappa \sqrt[4]{n}r)\text{e}^{\text{i}\ell(\theta-\phi)}\quad \text{and}\quad u_{\text{ev}}(r,\theta)=\sum_{|\ell|=0}^{\infty}\text{i}^{\ell}d_{\ell}J_{\ell}(\text{i}\kappa \sqrt[4]{n} r)\text{e}^{\text{i}\ell(\theta-\phi)},
\]
where $J_{\ell}$ is the Bessel function of the first kind. We notice that we have 
$$\Delta u_{\text{pr}}=-\kappa^2\sqrt{n}u_{\text{pr}} \quad \text{ and } \quad \Delta u_{\text{ev}}=\kappa^2\sqrt{n}u_{\text{ev}}.$$ Therefore, the four boundary conditions at $r=\epsilon$ impose a linear system of equations given by
\[
\mathbb M\mathbf{u}=\mathbf{f},
\]
where
\[
\mathbb{M} =
\begin{bmatrix}
    H_{\ell}^{(1)}(\kappa \varepsilon) &
    H_{\ell}^{(1)}(\text{i}\kappa \varepsilon) &
    -J_{\ell}(\kappa\sqrt[4]{n}\,\varepsilon) &
    -J_{\ell}(\text{i}\kappa\sqrt[4]{n}\,\varepsilon)\\[8pt]

    \kappa H_{\ell}^{(1)\prime}(\kappa \varepsilon) &
    \text{i}\kappa H_{\ell}^{(1)\prime}(\text{i}\kappa \varepsilon) &
    -\kappa \sqrt[4]{n}\,J_{\ell}^\prime(\kappa\sqrt[4]{n}\,\varepsilon) &
    -\text{i}\kappa \sqrt[4]{n}\,J_{\ell}^\prime(\text{i}\kappa\sqrt[4]{n}\,\varepsilon)\\[8pt]

    -\kappa^2 H_{\ell}^{(1)}(\kappa \varepsilon) &
    \kappa^2 H_{\ell}^{(1)}(\text{i}\kappa \varepsilon) &
    \kappa^2\sqrt{n}\,J_{\ell}(\kappa\sqrt[4]{n}\,\varepsilon) &
    -\kappa^2\sqrt{n}\,J_{\ell}(\text{i}\kappa\sqrt[4]{n}\,\varepsilon)\\[8pt]

    -\kappa^3 H_{\ell}^{(1)\prime}(\kappa \varepsilon) &
    \text{i}\kappa^3 H_{\ell}^{(1)\prime}(\text{i}\kappa \varepsilon) &
    \kappa^3 (\sqrt[4]{n})^3 J_{\ell}^\prime(\kappa\sqrt[4]{n}\,\varepsilon) &
    -\text{i}\kappa^3 (\sqrt[4]{n})^3 J_{\ell}^\prime(\text{i}\kappa\sqrt[4]{n}\,\varepsilon)
\end{bmatrix}
\]
and
\[
\mathbf{u}=\begin{bmatrix}
    a_{\ell}\\
    b_{\ell}\\
    c_{\ell}\\
    d_{\ell}
\end{bmatrix} \quad \text{with}\quad \mathbf{f}=\begin{bmatrix}
    -J_{\ell}(\kappa \epsilon)\\
    -\kappa J_{\ell}^{\prime}(\kappa \epsilon)\\
    \kappa^2 J_{\ell}(\kappa \epsilon)\\
    \kappa^3 J_{\ell}^{\prime}(\kappa\epsilon)
\end{bmatrix}
\]
for every $\ell\in \mathbb Z$. 

This implies that, by solving the system for the Fourier coefficients (in particular $a_{\ell}$) we have the biharmonic far--field pattern corresponding to the disks of radius $\epsilon$ is given by
\[
u^{\infty}(\theta,\phi)=\frac{4}{\text{i}}\sum_{|\ell|=0}^{\infty}a_{\ell}\text{e}^{\text{i}\ell(\theta-\phi)}
\]
from the asymptotics of the Hankel functions. We will compare this exact far--field pattern with the Born approximation $u_{\text{B}}^{\infty}(\theta,\phi)$ given by
\[
u_{\text{B}}^{\infty}(\theta,\phi)=\frac{\kappa^2}{2}\int_{B_{\epsilon}}(n-1)\text{e}^{-\text{i}\kappa y \cdot (\hat{x}-d)}\,\text{d}y
\]
where we use the fact that 
$$\hat{x}=(\cos{\theta},\sin{\theta}) \quad \text{and} \quad d=(\cos{\phi},\sin{\phi}).$$
We truncate the infinite series for $u^{\infty}(\theta,\phi)$ taking only the values $|\ell| \leq 10$, corresponding to a total of $21$ terms in the truncated series.

Here, consider the far--field matrices 
\[
\mathbf{F}_{\text{exact}}=\left[u^{\infty} (\theta_i,\phi_j) \right]_{i,j=1}^{64}\quad \text{and}\quad \mathbf{F_{\text{Born}}}=\left[u_{\text{B}}^{\infty}(\theta_i,\phi_j)\right]_{i,j=1}^{64}
\]
with $\theta_i=2\pi(i-1)/64$ and $\phi_j=2\pi(j-1)/64$ for $i,j=1,\dots,64$. We find that the absolute error
\[\|\mathbf{F}_{\text{exact}} - \mathbf{F}_{\text{Born}}\|_{\infty}\]
for small scatterers $B_{\epsilon}$ decreases slightly for smaller $\epsilon$. Table 1 displays the relative errors computed for a fixed wavenumber $\kappa = 2$, a fixed refractive index $n=1.0+0.5\text{i}$, and $64$ incident and observation directions. Thus, the Born approximation proves relatively accurate for weak scatterers, e.g., disks with radii given by small $\epsilon>0$. This would imply that the Born approximation for the biharmonic scattering problem is a constant multiple to the equivalent approximation of acoustic scatterers. Even with this being the case, we will provide a few numerical examples using the far--field data computed via the Born approximation. It is clear from previous works that this data can recover weak scatterers but we still provide the examples for completeness.  

\begin{table}[h!]\label{first_table}
\centering
\begin{tabular}{c|c}
\hline
$\epsilon$ & $\|\mathbf{F}_{\text{exact}} - \mathbf{F}_{\text{Born}}\|_{\infty}$ \\
\hline
1.00 &  0.6480\\
0.90 & 0.5008\\
0.80&  0.3784 \\
0.70 & 0.2774\\
0.60 & 0.1935\\
0.50 & 0.2134 \\
\hline
\end{tabular}
\caption{Absolute error between the exact and Born far-field matrices for small disks $B_{\epsilon}$.}
\label{tab:born_error}
\end{table}

\begin{figure}[h!]
\centering
\begin{minipage}{0.48\textwidth}
    \centering
    \includegraphics[width=\textwidth]{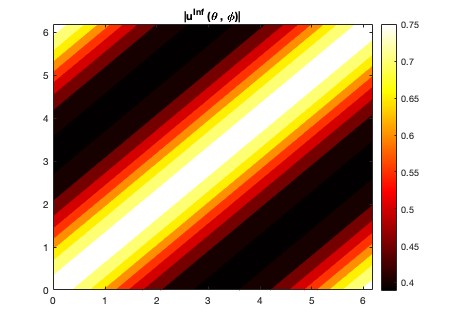} 
\end{minipage}
\hfill
\begin{minipage}{0.48\textwidth}
    \centering
    \includegraphics[width=\textwidth]{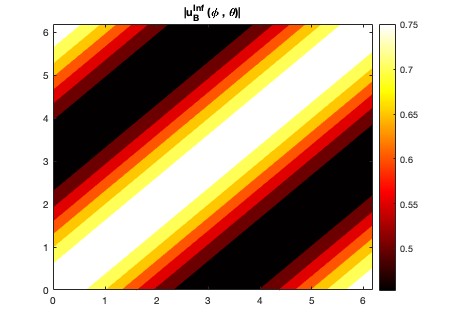} 
\end{minipage}
\caption{Exact and Born-approximated far-field matrices for scattering by a disk of radius $\epsilon=0.5$ with complex refractive index $n=1.0+0.5\text{i}$
at wavenumber $\kappa=2$. The comparison illustrates the close approximation with a slight deviation of the Born approximation from the exact solution in the presence of moderate absorption.}
\label{fig:disk_matrices}
\end{figure}

Figure \ref{fig:disk_matrices} compares the exact and Born--approximated far--field matrices for the disk of radius $\epsilon=0.5$, $n=1.0+1.5\text{i}$, and wavenumber $\kappa=2$. The Born and exact far--field matrices look strikingly similar, capturing similar angular scattering patterns.  Figure \ref{fig:disk_reconstruction} shows example reconstructions of a penetrable small disk with radius $\epsilon = 0.5$, wavenumber of $\kappa=2$, and refractive index $n=1.0+0.5\text{i}$.  Using the imaginary part of the exact biharmonic far--field pattern and its Born approximation as far--field data, we applied the factorization method, sampling the computational domain $[-2,2] \times [-2,2]$ on a $200 \times 200$ grid.  The exact reconstruction captures finer features, while the Born approximation provides a qualitatively accurate approximation.

\begin{figure}[h!]
\centering
\begin{minipage}{0.48\textwidth}
    \centering
    \includegraphics[width=\textwidth]{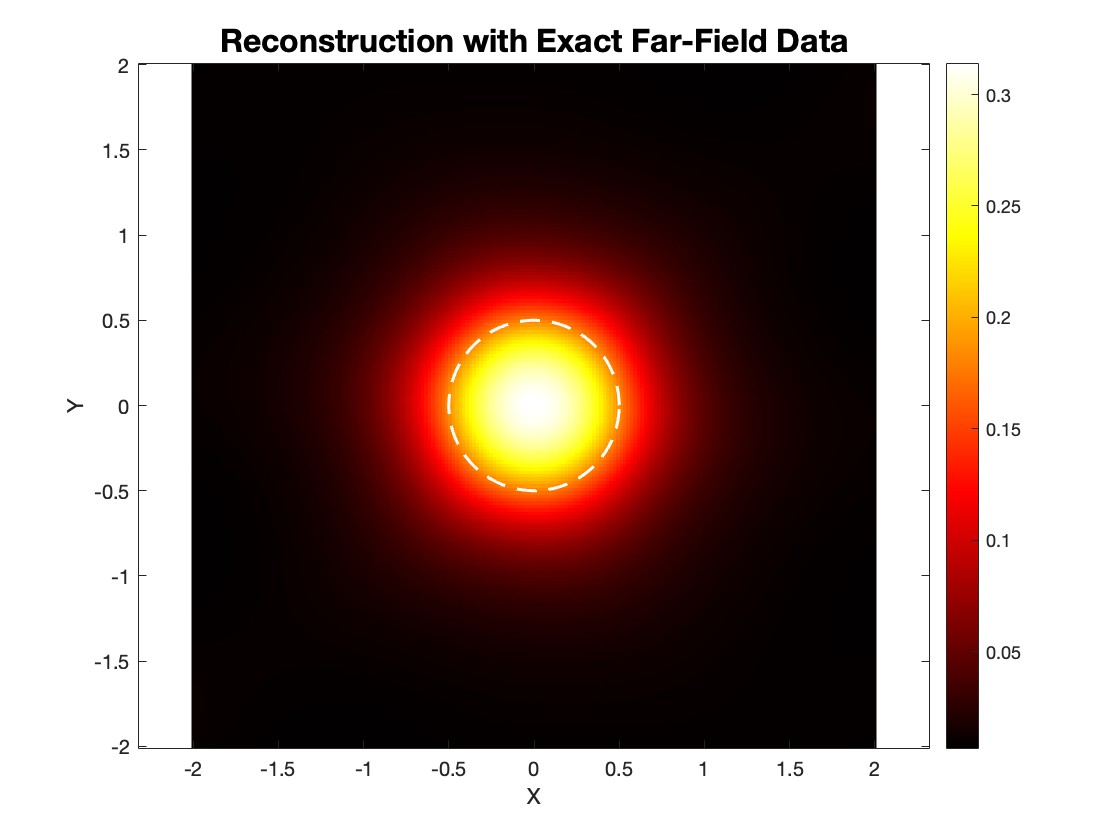} 
\end{minipage}
\hfill
\begin{minipage}{0.48\textwidth}
    \centering
    \includegraphics[width=\textwidth]{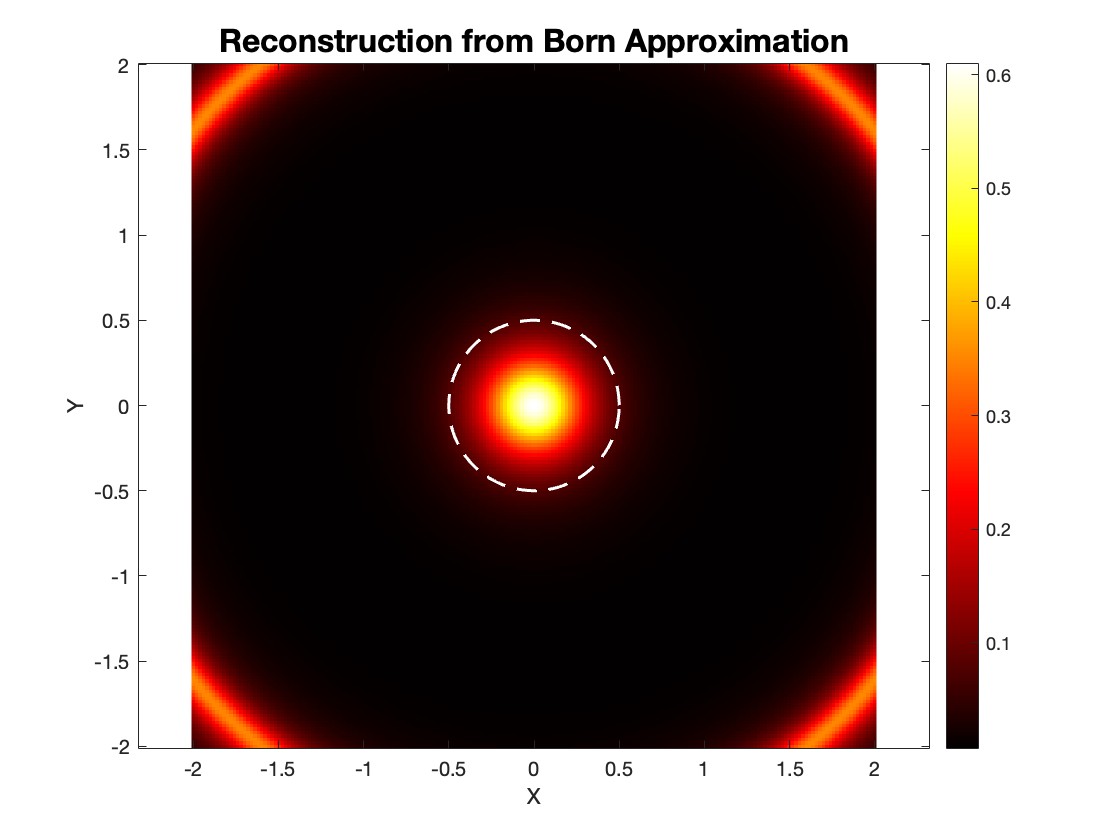} 
\end{minipage}
\caption{Comparison of the reconstructed penetrable disk $D=B_{\epsilon}$ using far-field data for the disk scatterer with radius $\epsilon = 0.5$, wavenumber $\kappa = 2$, and $n=1.0+0.5\text{i}$.}
\label{fig:disk_reconstruction}
\end{figure}
 
In order to present other numerical examples, we consider reconstructing two different shapes: a star--shaped region and a kite--shaped region, defined by
\[
\partial D=( x_1(t),x_2(t))^\top,\quad 0\leq t\leq 2\pi.
\]
Here we assume that the function $x_j(t)$ are twice continuously differentiable on $[0,2\pi]$ for $j=1,2$. In our numerical examples, we consider the star--shaped region which has a boundary that is parameterized by 
\[
\partial D = 0.175\Big(0.3\cos(5t)+2\Big)(\cos t, \sin t)^{\top}, \quad \text{ for  $t \in [0,2\pi)$}
\]
and the kite-shaped boundary that is given parametrically by
\[
\partial D=0.5\big(0.75\cos t + 0.3\cos 2t, \sin t\big)^\top, \quad \text{ for  $t \in [0,2\pi)$}.
\]
In these examples, we use fixed wavenumbers $\kappa=2\pi$ and $\kappa=3\pi$ with a constant contrast $n=2.5+0.5\text{i}$. The sampling region is $[-2,2]\times [-2,2]$, and we select $200\times 200$ equally spaced points within the region. We present contour plots of the imaging functions for both scatterers at different noise levels to assess the stability of our method. For our reconstructions, Tikhonov regularization with regularization parameter $\alpha = 10^{-5}$ was employed to stabilize the reconstructions. In all figures, the dotted white line represents the boundary $\partial D$ of the cavity.

\subsection{Example 1. A star--shaped cavity}
We present numerical reconstructions for the penetrable star--shaped
cavity using the factorization method imaging function $W_{\text{FM}}(z)$. Figure \ref{fig:star1_reconstruction} shows the reconstruction obtained with $W_{\text{FM}}$ with no error and fixed wavenumbers $\kappa=2\pi$ and $\kappa=3\pi$, respectively. Figure \ref{fig:star2_reconstruction} shows the reconstruction obtained with the imaging function $W_{\text{FM}}$ with noise levels given by $\delta=0.05$ and $\delta=0.10$, respectively, representing a significant amount of noise in the data. The left panel corresponds to the $5\%$ noise ($\delta = 0.05$), while the right panel shows the $10\%$ noise ($\delta = 0.10$). Despite the increase in noise, the differences in the imaging function are minor and the factorization method still produces accurate reconstructions, as highlighted by the white dashed line indicating the true boundary $\partial D$.  Figure \ref{fig:star3_reconstruction} shows reconstructions of the penetrable star-shaped scatterer for two different locations. The left panel corresponds to a small shift to the left of the origin at $(-0.5. 0.5)$ with the sampling region $[-2,2]\times [-2,2]$, while the right panel shows a shift to the right of the grid center at $(1,1.5)$ with a sampling region of $[-2,2]\times [-2,2]$ with $5\%$ noise added for both reconstructions. The reconstructions are accurate when shifted and the recovery of the penetrable star-shaped scatterer is robust towards noise. We note that here and all our other examples, that the white dashed line in the figures indicates the true boundary $\partial D$. This allows us to see that accuracy of the factorization method. 

\begin{figure}[h!]
\centering
\begin{minipage}{0.48\textwidth}
    \centering
    \includegraphics[width=\textwidth]{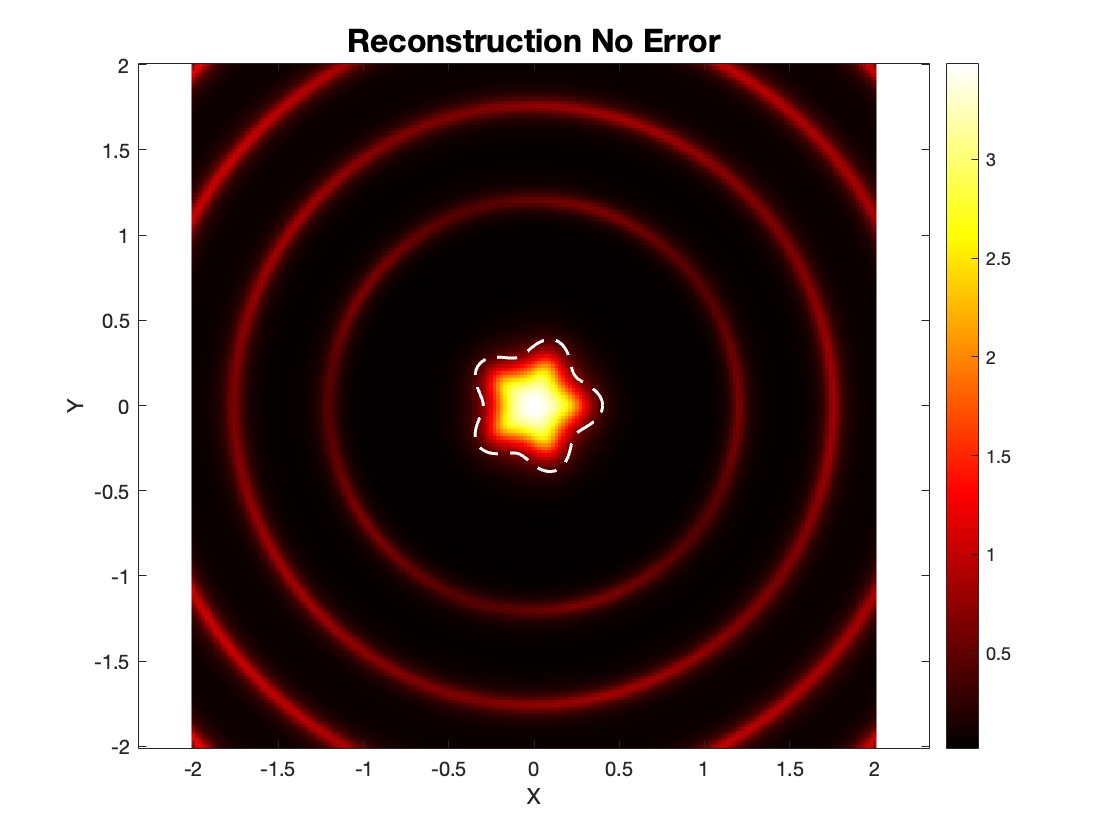} 
\end{minipage}
\hfill
\begin{minipage}{0.48\textwidth}
    \centering
    \includegraphics[width=\textwidth]{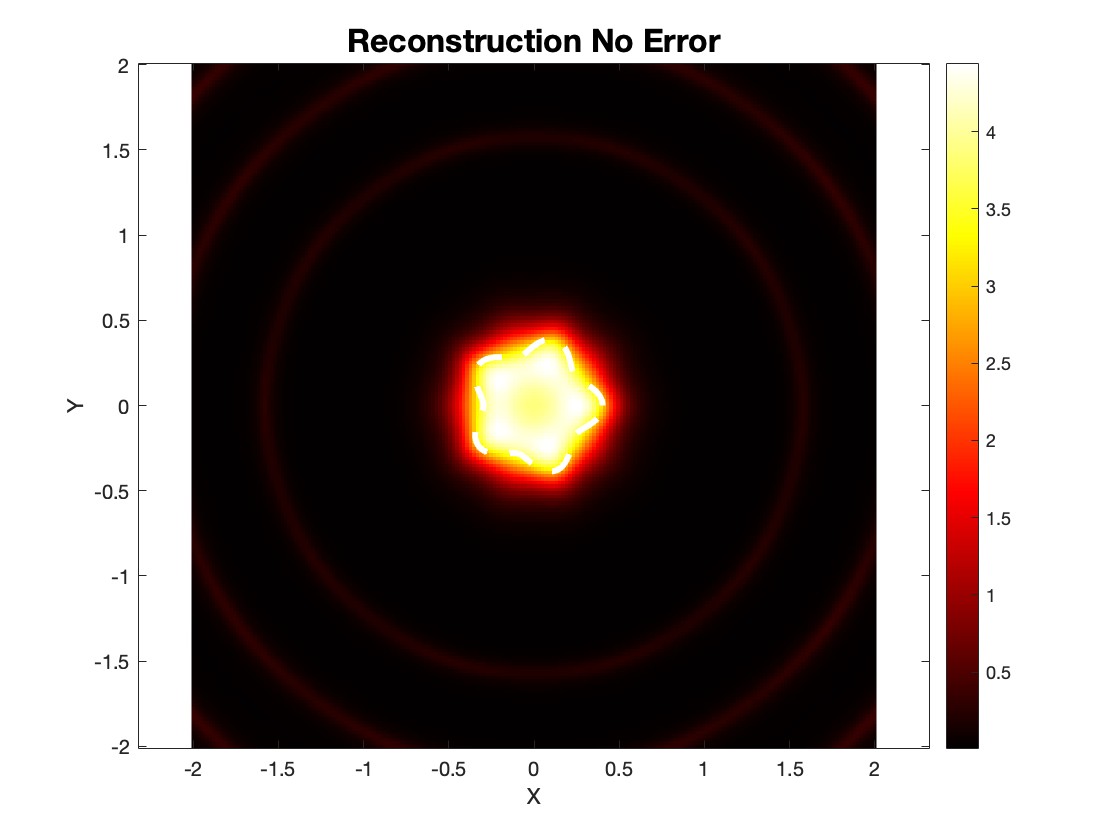} 
\end{minipage}
\caption{Reconstruction of the star--shaped penetrable cavity using the factorization method for two different wavenumbers: $\kappa = 2\pi$ (left) and $\kappa = 3\pi$ (right) with a constant contrast $n=2.5+0.5\text{i}$. }
\label{fig:star1_reconstruction}
\end{figure}

\begin{figure}[h!]
\centering
\begin{minipage}{0.48\textwidth}
    \centering
    \includegraphics[width=\textwidth]{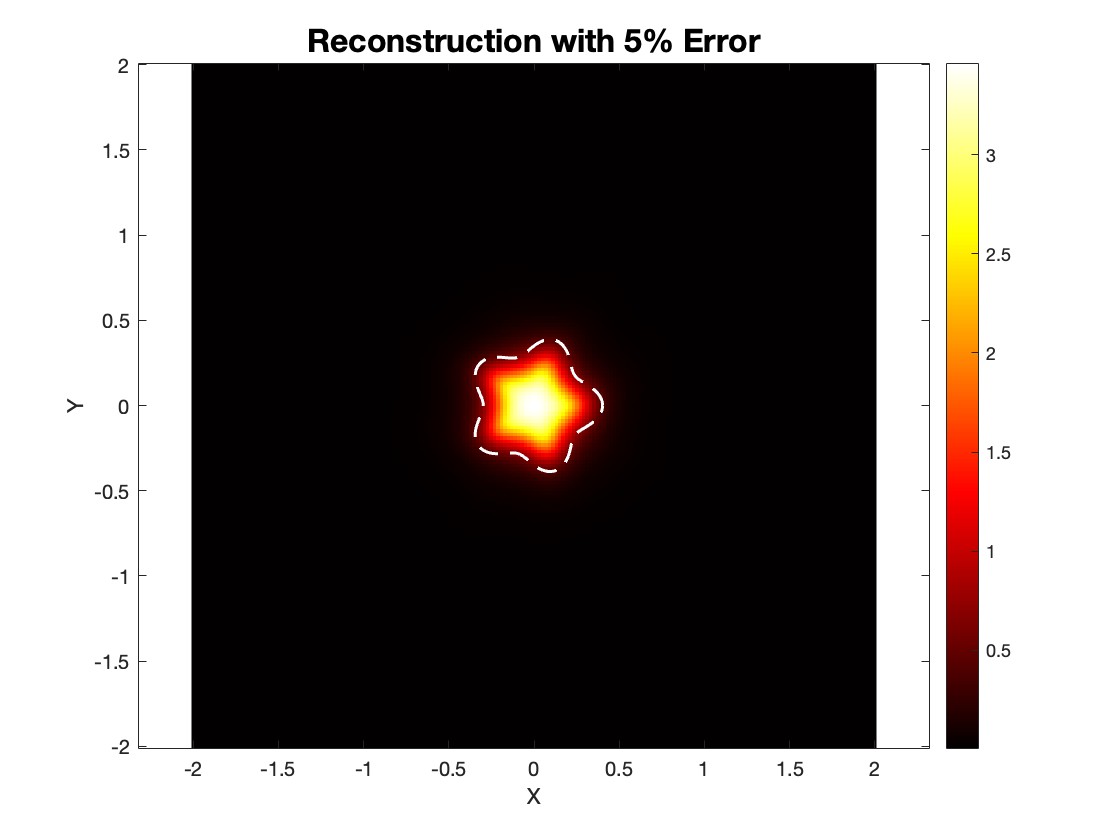} 
\end{minipage}
\hfill
\begin{minipage}{0.48\textwidth}
    \centering
    \includegraphics[width=\textwidth]{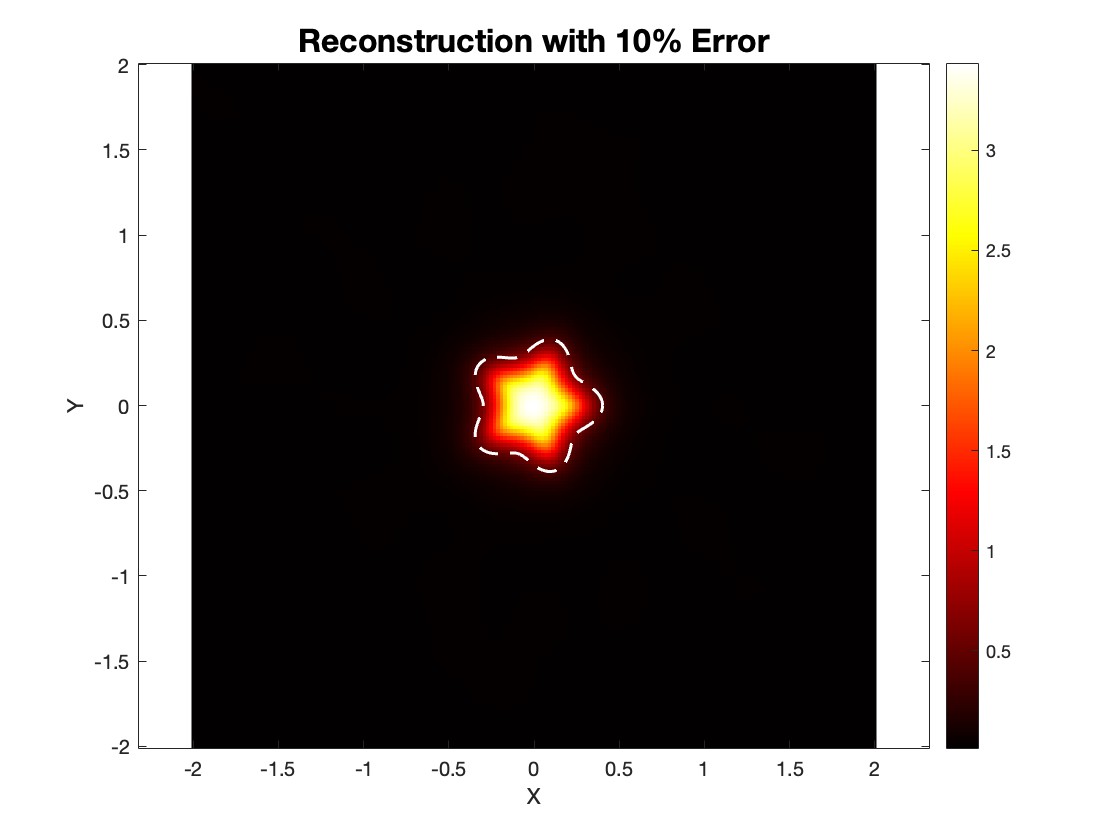} 
\end{minipage}
\caption{Reconstruction of the star--shaped penetrable cavity using the factorization method for $\kappa = 2\pi$ with noise added. Left: 5\% noise ($\delta = 0.05$). Right: 10\% noise ($\delta = 0.10$).}
\label{fig:star2_reconstruction}
\end{figure}

\begin{figure}[h!]
\centering
\begin{minipage}{0.48\textwidth}
    \centering
    \includegraphics[width=\textwidth]{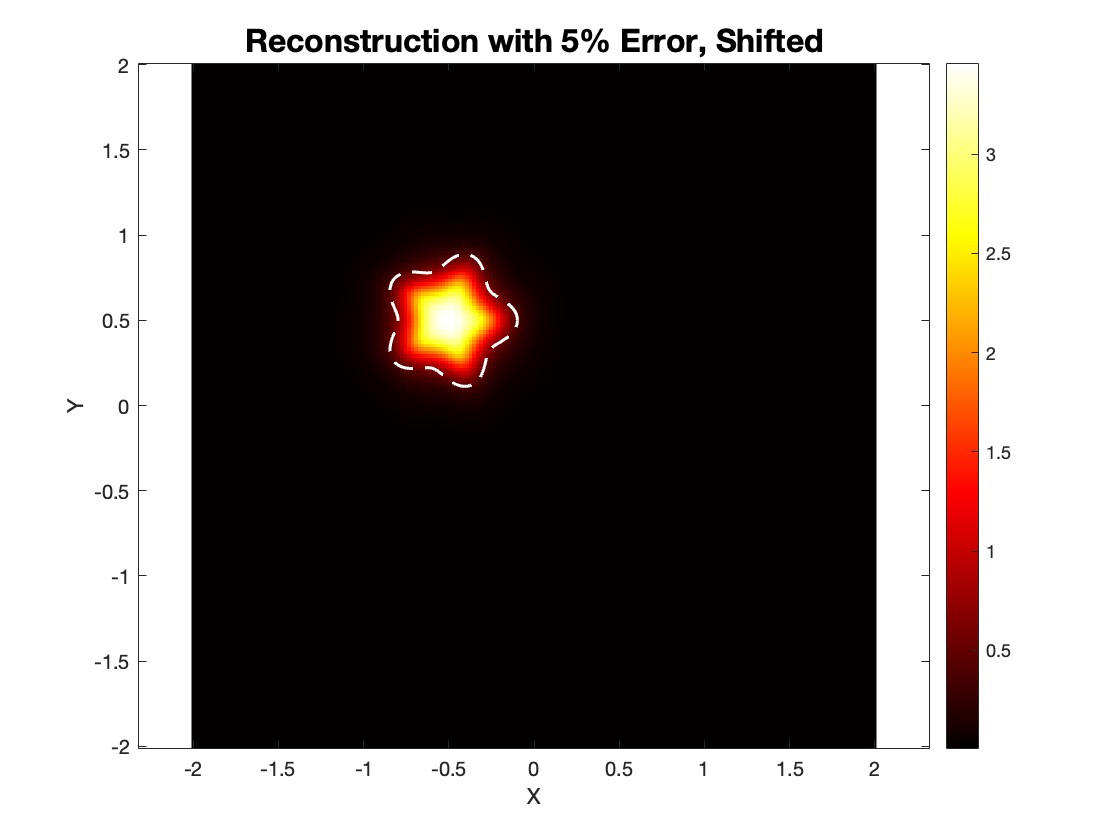} 
\end{minipage}
\hfill
\begin{minipage}{0.48\textwidth}
    \centering
    \includegraphics[width=\textwidth]{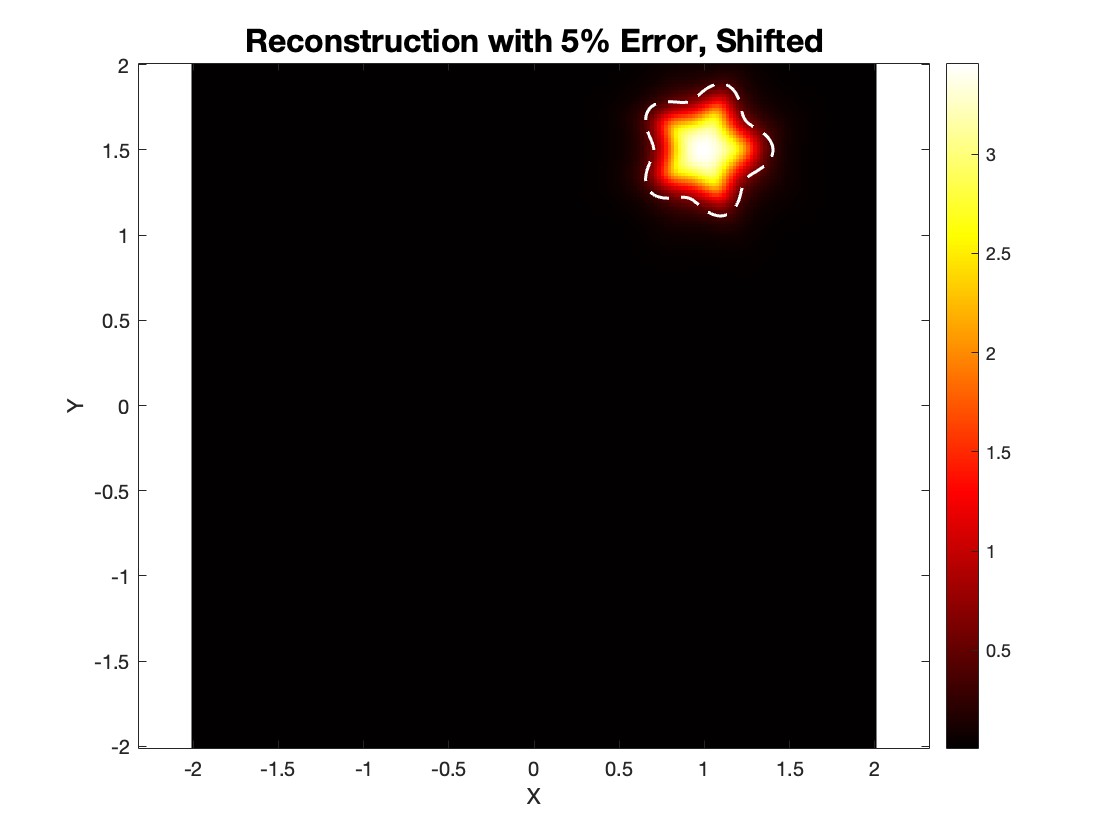} 
\end{minipage}
\caption{Reconstruction of the shifted star--shaped penetrable cavity using the factorization method for $\kappa = 2\pi$ with $5\%$ noise and index of refraction $n=2.5+0.5\text{i}$. Left: Scatterer shifted to $(-0.5, 0.5)$. Right: Scatterer shifted to $(1,1.5)$.}
\label{fig:star3_reconstruction}
\end{figure}

\subsection{Example 2. A kite--shaped cavity.}
We now present another numerical example using the penetrable kite--shaped cavity using the factorization method imaging function $W_{\text{FM}}(z)$.
For these reconstructions, we used the same physical parameters as in the star-shaped cavity. In Figure \ref{fig:kite1_reconstruction}, we observe that the reconstructions of the kite--shaped cavity are reasonably accurate, with slightly improved spatial resolution as the wavenumber $\kappa$ increases. 

Figure \ref{fig:kite2_reconstruction} shows that the reconstructions remain robust in the presence of random noise, highlighting the effectiveness of the factorization method in identifying the penetrable cavity. As shown in Figure \ref{fig:kite3_reconstruction}, the factorization method successfully recovers the size, shape, and location of the kite--shaped cavity. The left panel corresponds to a small shift to the left of the origin at $(-0.5. 0.5)$ with the sampling region $[-2,2]\times [-2,2]$, while the right panel shows a shift to the right of the grid center at $(1,1.5)$ with a sampling region of $[-2,2]\times [-2,2]$ with $5\%$ noise added for both reconstructions. 

\begin{figure}[h!]
\centering
\begin{minipage}{0.48\textwidth}
    \centering
    \includegraphics[width=\textwidth]{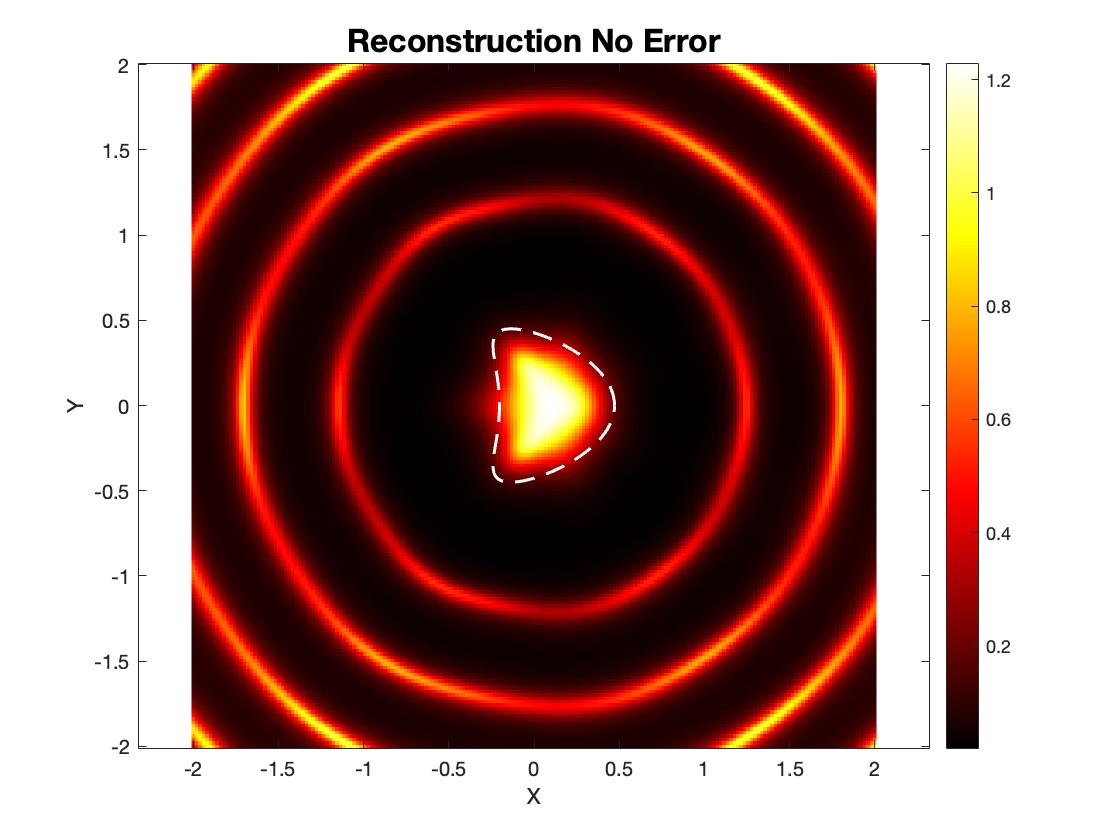} 
\end{minipage}
\hfill
\begin{minipage}{0.48\textwidth}
    \centering
    \includegraphics[width=\textwidth]{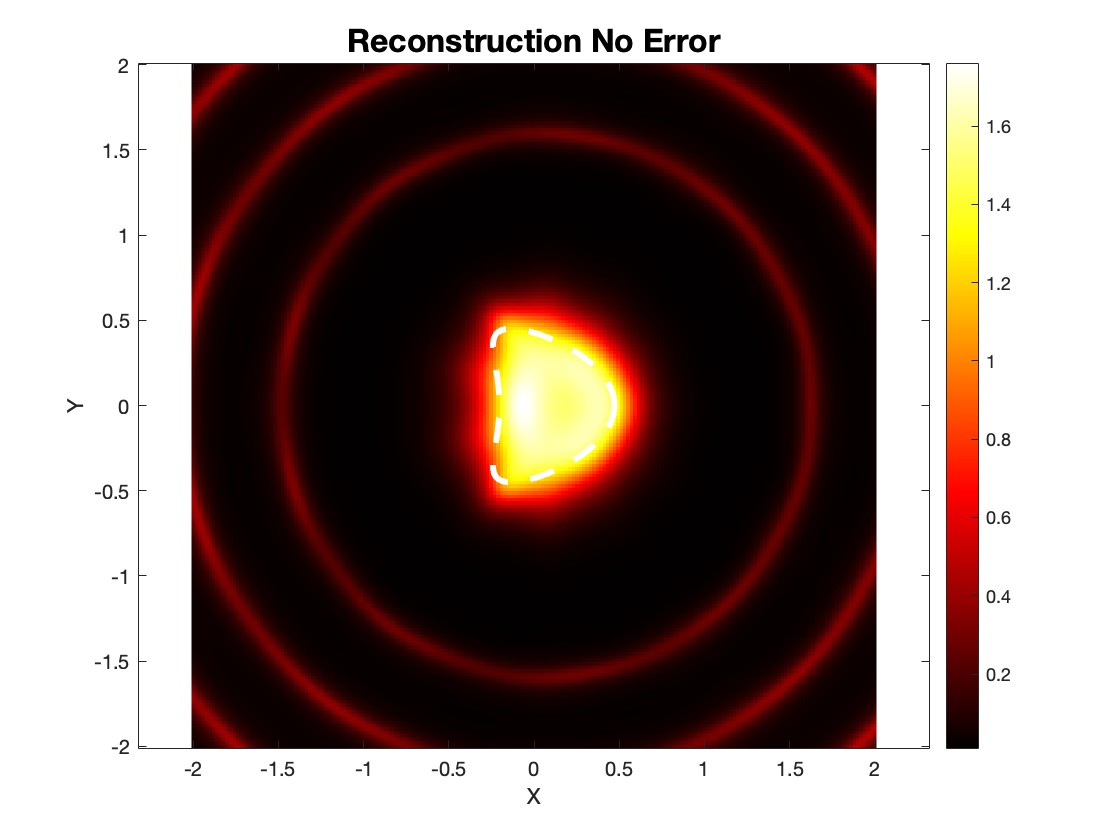} 
\end{minipage}
\caption{Reconstruction of the penetrable kite--shaped cavity using the factorization method for two different wavenumbers: $\kappa = 2\pi$ (left) and $\kappa = 3\pi$ (right) with a constant contrast $n=2.5+0.5\text{i}$. }
\label{fig:kite1_reconstruction}
\end{figure}

\begin{figure}[h!]
\centering
\begin{minipage}{0.48\textwidth}
    \centering
    \includegraphics[width=\textwidth]{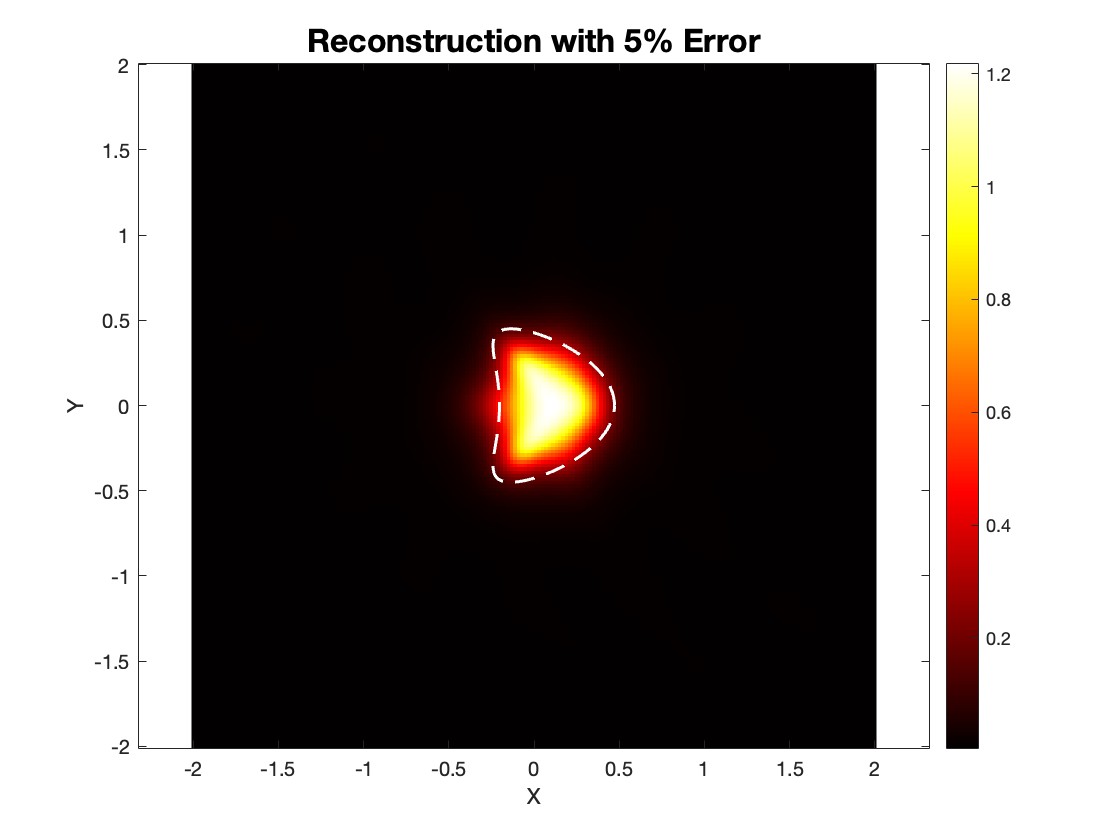} 
\end{minipage}
\hfill
\begin{minipage}{0.48\textwidth}
    \centering
    \includegraphics[width=\textwidth]{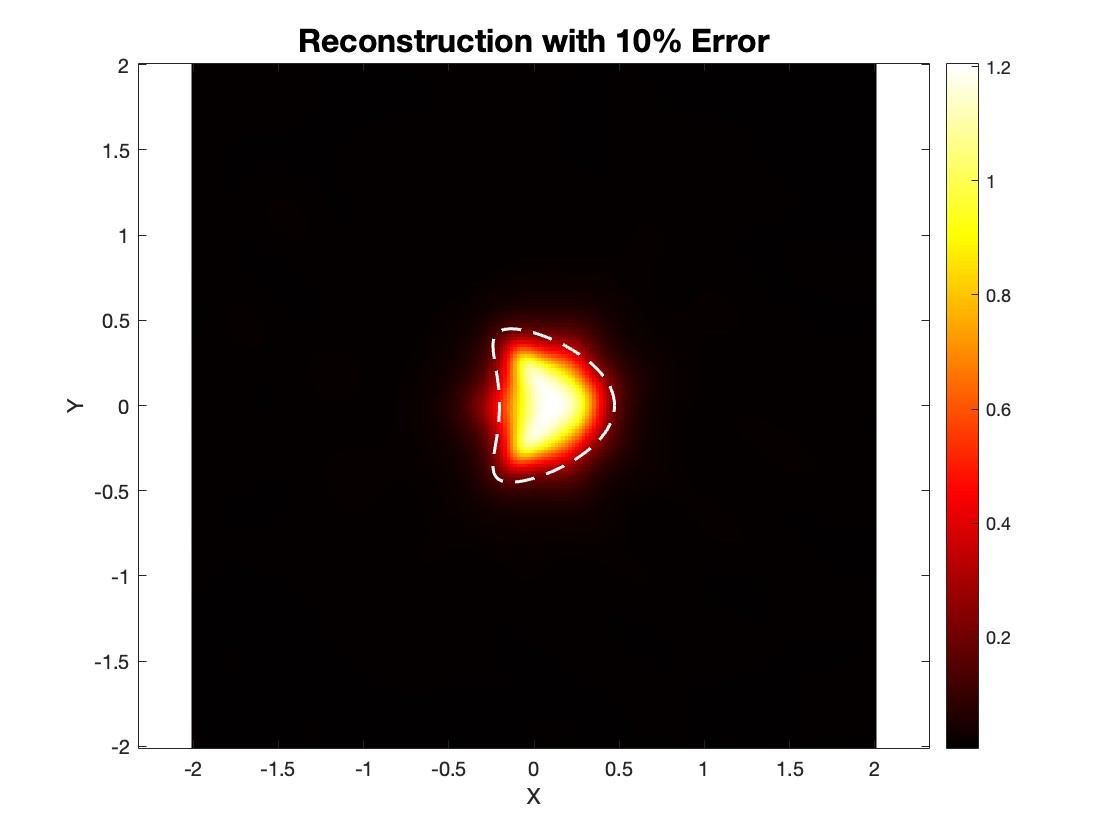} 
\end{minipage}
\caption{Reconstruction of the penetrable kite--shaped cavity using the factorization method for $\kappa = 2\pi$ with noise added. Left: 5\% noise ($\delta = 0.05$). Right: 10\% noise ($\delta = 0.10$).}
\label{fig:kite2_reconstruction}
\end{figure}

\begin{figure}[h!]
\centering
\begin{minipage}{0.48\textwidth}
    \centering
    \includegraphics[width=\textwidth]{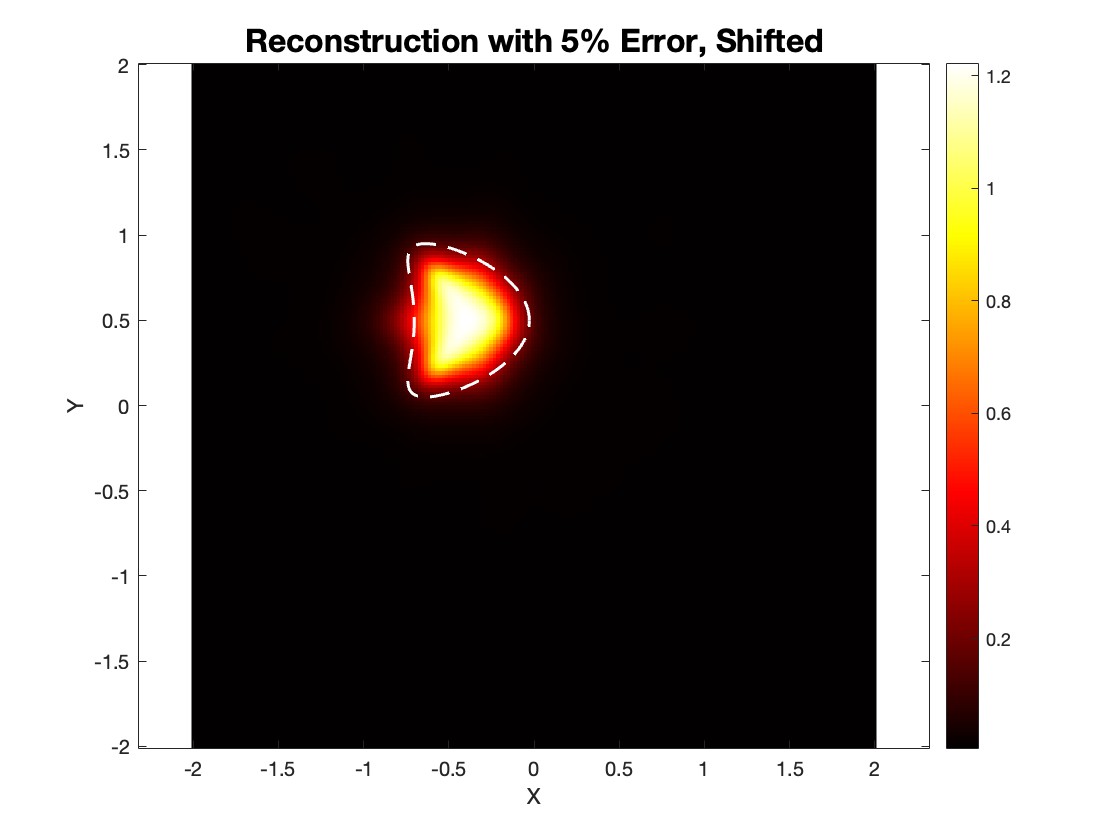} 
\end{minipage}
\hfill
\begin{minipage}{0.48\textwidth}
    \centering
    \includegraphics[width=\textwidth]{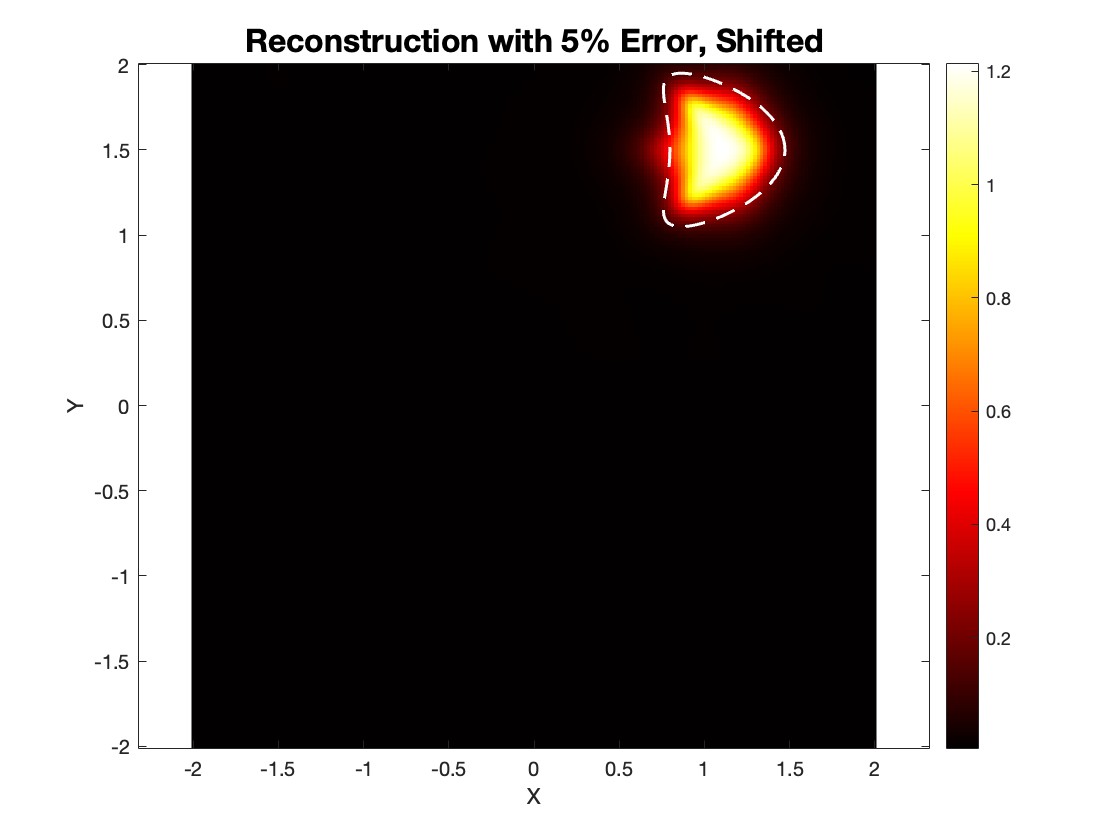} 
\end{minipage}
\caption{Reconstruction of the shifted kite--shaped penetrable cavity using the factorization method for $\kappa = 2\pi$ with $5\%$ noise and index of refraction $n=2.5+0.5\text{i}$. Left: Scatterer shifted to $(-0.5, 0.5)$. Right: Scatterer shifted to $(1,1.5)$.}
\label{fig:kite3_reconstruction}
\end{figure}

\subsection{Example 3. Two small disks.}
For our last numerical example, we consider the reconstruction of two disks. This example shows that the factorization method is capable of recovering two scatterers. Just as before, we will use the Born approximation to compute the far--field data. To this end, we assume that the boundary of the scatterer is given by $\partial D= \partial D_1 \cup \partial D_2$ where 
$$\partial D_1 = 0.5(\cos t, \sin t)^{\top} + (-0.5, 0.5)^\top \quad \text{ and } \quad \partial D_1 = 0.5(\cos t, \sin t)^{\top} + (1, 1.5)^\top$$
for $t \in [0,2\pi)$. Using the parameters $\kappa=2\pi$ and $n=2.5+0.5\text{i}$ as in the previous examples we present the reconstructions. We present reconstructions with both noiseless and 5$\%$ noisy far--field data. By again sampling over $[-2,2]\times [-2,2]$ with $200\times 200$ equally spaced points sampling points. In Figure \ref{fig:2disk_reconstruction}, we see the reconstructions for this example.  

\begin{figure}[h!]
\centering
\begin{minipage}{0.48\textwidth}
    \centering
    \includegraphics[width=\textwidth]{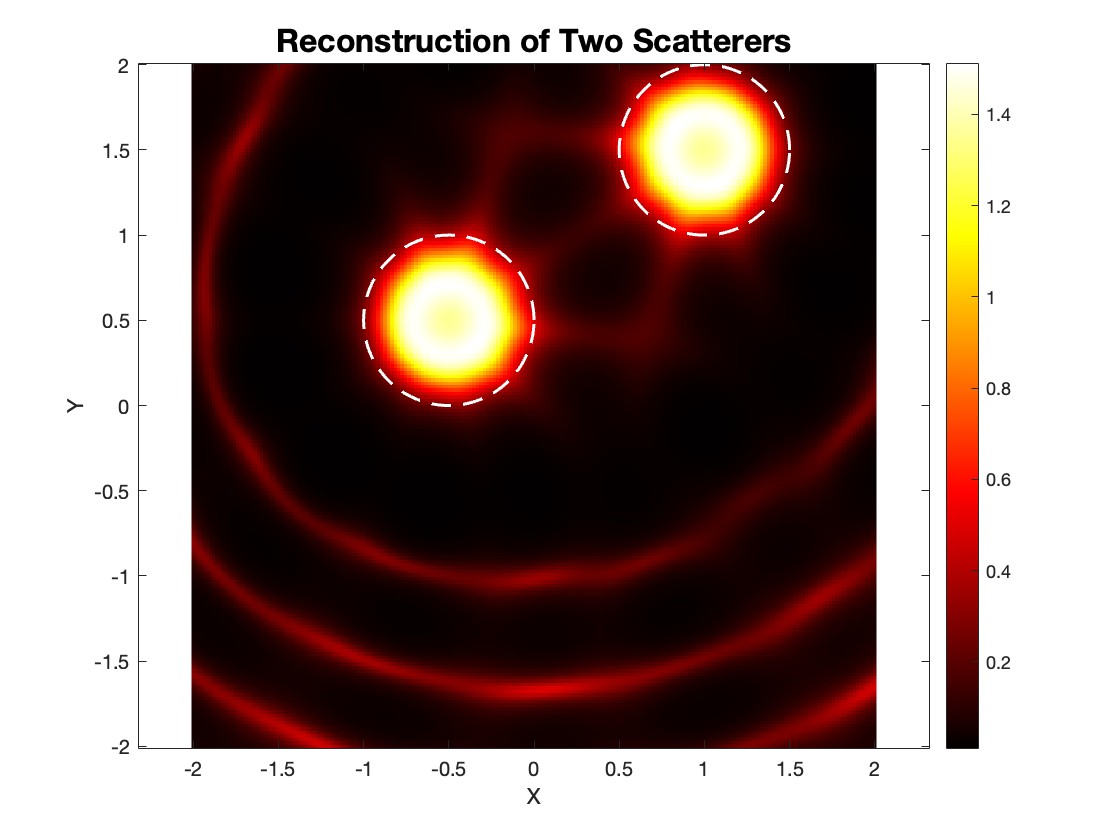} 
\end{minipage}
\hfill
\begin{minipage}{0.48\textwidth}
    \centering
    \includegraphics[width=\textwidth]{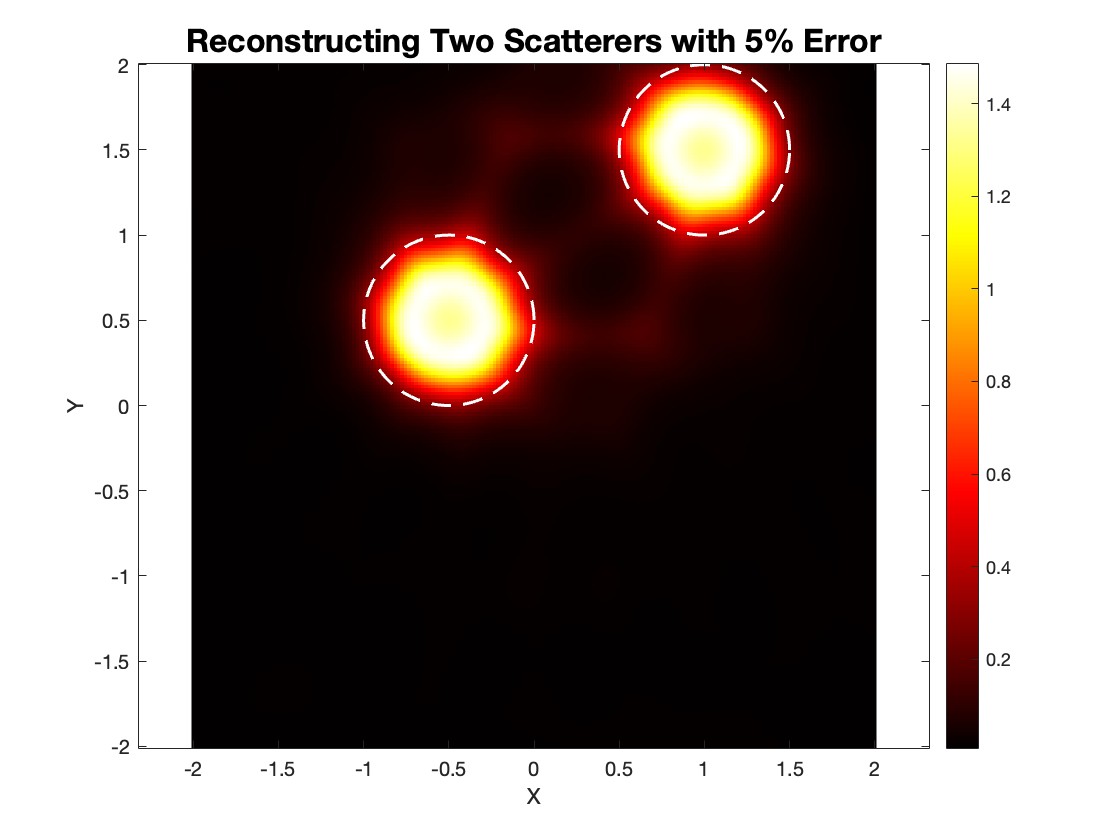} 
\end{minipage}
\caption{Reconstruction of the penetrable scatterer given by two disks using the factorization method for $\kappa = 2\pi$ and index of refraction $n=2.5+0.5\text{i}$. Left: 0\% noise ($\delta = 0.00$). Right: 5\% noise ($\delta = 0.05$).}
\label{fig:2disk_reconstruction}
\end{figure}

\section{Conclusion}\label{section_conclusion}
In this paper, we presented a theoretical justification of the factorization method for the inverse biharmonic scattering problem for a penetrable, absorbing scatterer, thereby providing a rigorous analytical framework for reconstructing transmission-type obstacles in Kirchhoff--Love plates. Building on recent advances in sampling methods for biharmonic waves, our formulation establishes a spectral characterization of the far--field operator associated with the transmission problem and demonstrates that the support of the scatterer can be identified from the range of a self-adjoint operator from the far--field data. In addition, we analyzed the Born approximation for weak scatterers within the biharmonic setting and evaluated its performance through numerical experiments, confirming the reliability of the approximation for weak scatterers. Numerical experiments were performed to demonstrate the effectiveness of our method. The inclusion of an absorbing condition plays a key role, as it guarantees the uniqueness and well--posedness of the associated direct problem (\ref{eqn1})-(\ref{SRC}) (see \cite{CejaAyalaHarrisSanchezVizuet2025}) while ensuring the coercivity property of the auxiliary middle operator $\Im{(T)}$ required in the factorization method. 

There is future work to be done on extending the factorization method analysis to impenetrable obstacles like clamped or simply supported cavities, where the factorization method becomes more intricate. In contrast to the transmission case considered here, the corresponding direct problems involve boundary conditions that require a careful treatment through boundary integral formulations, and the associated far--field operators no longer exhibit the same factorization structure. Another point of significant practical interest will be to investigate the application of sampling methods to the transmission problem under limited--aperture data, as most physical experiments are restricted to partial observations of the scattered field. Finally, inverse near--field scattering has not yet been applied to the biharmonic transmission problem. Developing corresponding sampling based inversion with measured near--field data provides another direction for future work.


\end{document}